\documentclass[review]{elsarticle}

\usepackage{lineno,hyperref}

\journal{Transportation Science}
\usepackage{multirow}
\usepackage{etex}
\usepackage{amsfonts}
\usepackage{mathrsfs}
\usepackage{comment}
\usepackage{tablefootnote}
\usepackage{float}
\usepackage{algorithm}
\usepackage{algpseudocode}
\usepackage{todonotes}
\usepackage{amsmath}
\numberwithin{equation}{section}
\usepackage{longtable}
\usepackage{threeparttable}
\usepackage{caption}
\usepackage{pdflscape}
\usepackage{graphics}
\usepackage{subfigure}
\usepackage{subcaption}
\usepackage{ulem}
\usepackage{soul}

\usepackage{xspace}
\newcommand\BRFrag{battery-restricted fragment\xspace}
\newcommand\BRFrags{battery-restricted fragments\xspace}
\newcommand\SOS{checker\xspace}
\newcommand\FBC{SaC\xspace}
\newcommand\SOSfull{Fragment-Based Checker\xspace}
\newcommand\FBCfull{Select-and-Check\xspace}

\usepackage{microtype}
\usepackage{amsthm}
\usepackage{amsmath}

\usepackage{amssymb}
\usepackage{stmaryrd}
\usepackage{color,xcolor}
\usepackage{soul}
\soulregister{\cite}7
\soulregister{\citep}7
\soulregister{\ref}7

\usepackage{graphicx}
\usepackage{geometry}
\usepackage{tikz}
\usetikzlibrary{shapes.geometric, arrows}

\usepackage{footnote}
\makesavenoteenv{table}
\newcommand{\beq}{\begin{equation}}
\newcommand{\eeq}{\end{equation}}

\theoremstyle{definition}
\newtheorem{exmp}{Example}
\theoremstyle{remark}

\theoremstyle{TH}
\newtheorem{theorem}{Theorem}
\newtheorem{defn}{Definition}

\newtheorem{observation}{Observation}

\makeatletter 
\def\@makecaption#1#2{%
  \vskip\abovecaptionskip
  \sbox\@tempboxa{#1 #2}%
    {\bfseries #1} #2\par
  \vskip\belowcaptionskip}
\makeatother
\usepackage{booktabs}
\usepackage[justification=centering]{caption}

\usepackage{hyperref}
\hypersetup{
    colorlinks=true,
    linkcolor=blue,
    filecolor=gray,      
    urlcolor=blue,
    citecolor=blue,
}
\biboptions{authoryear}

\allowdisplaybreaks

\begin{document}

\begin{frontmatter}
\title{The Bi-objective Electric Autonomous Dial-a-Ride Problem}
\author[address1]{Yue Su \corref{correspondingauthor}}
\ead{yue.su@inria.fr}
\author[address3]{Sophie N. Parragh}
\ead{sophie.parragh@jku.at}
\author[address2]{Nicolas Dupin}
\ead{nicolas.dupin@univ-angers.fr}
\author[address4,address5]{Jakob Puchinger}
\ead{jpuchinger@em-normandie.fr}

\address[address1]{Univ. Lille, Inria, CNRS, Centrale Lille, UMR 9189 CRIStAL, F-59000 Lille, France.}
\address[address2]{Univ Angers, LERIA, SFR MATHSTIC, F-49000 Angers, France.}
\address[address3]{Institute of Production and Logistics Management, Johannes Kepler University Linz, 4040, Linz, Austria.}
\address[address4]{EM Normandie Business School, Métis Lab, 92110, Clichy, France.}
\address[address5]{Université Paris-Saclay, CentraleSupélec, Laboratoire Génie Industriel, 91190, Gif-sur-Yvette, France.}

\cortext[correspondingauthor]{Corresponding author: yue.su@inria.fr}

\begin{abstract}
The electric autonomous dial-a-ride problem (E-ADARP) introduces electric, autonomously driving vehicles and their unique requirements into the classic dial-a-ride problem, where people are transported between pickup and drop-off locations. Next to an electric autonomous vehicle fleet, in the literature, a weighted-sum objective function, which combines the classic routing cost-oriented objective with a user-oriented objective function, has usually been considered. The user-oriented objective function minimizes the total excess user ride time. In this work, we treat them as two separate objective functions, which are optimized concurrently. In order to address the resulting bi-objective E-ADARP, we develop a novel exact framework (called fragment-based checker), whose core part is a smart ``select-and-check" algorithm that iteratively constructs feasible solutions using fragments. Several enhancements are proposed to enforce the computational efficiency of the proposed method. In the computational experiments, we evaluate several variants of our checker algorithm by leveraging a previously developed branch-and-price algorithm. We benchmark the checker-based framework against state-of-the-art criterion space frameworks (e.g., the $\epsilon$-constraint method and the balanced-box method), as well as a generalized branch-and-price algorithm. Numerical results on both bi-objective DARP and E-ADARP instances demonstrate the effectiveness of the proposed framework. With our proposed approaches, 21 out of 38 instances are solved optimally, where small-to-medium-sized instances are solved within seconds. On larger-scale instances, especially those requiring high battery end levels are computationally challenging to solve, our approaches provide high-quality approximations of the Pareto frontiers. Efficient solutions with varying energy restrictions are compared and we obtain valuable managerial insights for different kinds of service providers.
\end{abstract}


\end{frontmatter}

\section{Introduction}
More and more cities integrate ride-sharing or pooling services into their public transportation networks. Since rides are pooled and not served individually, like in a standard taxi-type setting, these services can be offered at a lower cost. However, when put into operation, they also need to be attractive (enough) from a user perspective. A possible way to measure user inconvenience in the context of ride-pooling is the length of a user's detour in comparison to a direct (taxi-type) ride. So, from an operator perspective, there are two objectives to consider: the costs of the service and the total excess user ride time, which measures this detour. When the electric autonomous dial-a-ride problem (E-ADARP) was first introduced by \cite{bongiovanni2019electric}, the authors already had these two objectives in mind and combined them in a weighted-sum objective function. In this work, we now consider them as two separate objectives that we optimize concurrently. Since the two goals are conflicting in the sense that lower excess user ride times usually lead to higher costs and vice versa, there will be no single solution that is optimal for both objectives. On the contrary, instead of one solution, a set of solutions will be optimal, each representing a different trade-off. The solutions forming this set are usually called efficient solutions, and their image in objective space is usually referred to as the Pareto frontier. Our goal is to compute this frontier and at least one pre-image (or efficient solution) per non-dominated point.

In order to address the bi-objective electric autonomous dial-a-ride problem (BO-EADARP), we build on the idea of battery-restricted fragments \citep{su2023deterministic} and fragment-based graph \citep{su2024branch}.
In the fragment-based representation of paths, an E-ADARP route is regarded as a concatenation of fragments (represented by a sequence of pickup and drop-off nodes), visits to charging stations (if required), and arcs from/to depots. Using pre-computed fragments, each associated with all possible minimum excess-user-ride-time schedules, allows for the exact computation of the minimum excess user ride times per route (as proved in \cite{su2023deterministic}). Furthermore, since a charging operation is only possible whenever there are no users on board the vehicle, visits to charging stations can only be planned before or after a fragment. In the branch-and-price algorithm of \citet{su2024branch} for the single-objective E-ADARP, this can be efficiently handled in the pricing step by keeping track of all possible trade-offs between the amount that can be recharged and the time necessary to do so \citep{desaulniers2016exact}. 
In this work, we solve the BO-EADARP, where the total travel time and the total excess user ride time are treated as separate objectives. Our work offers the following contributions: 
\begin{enumerate}
    \item Building on the ideas of fragments that can be concatenated to form routes, we develop a novel formulation of the BO-EADARP and a tailored objective-space search framework (the \SOS algorithm). This framework is based on a smart ``select-and-check" process and appears to be the first exact algorithm that produces the trade-off frontier between total travel time and (excess) user ride time for the E-ADARP.
    \item Several enhancement techniques (including a tightened formulation) are employed to enforce the computational efficiency of the proposed framework.
    \item Our proposed framework is benchmarked against available frameworks for multi-objective optimization, such as the well-known $\epsilon$-constraint approach, as well as two more recently developed frameworks \citep{boland2015criterion,parragh2019branch}. Extensive numerical experiments highlight the efficiency of the proposed framework.
    \item Managerial insights for different kinds of service providers are obtained from analyzing Pareto-optimal solutions.
\end{enumerate}

This paper is organized as follows. We first review the relevant literature in Section~\ref{sec::literature}. Then, we present a detailed description and a mathematical formulation of the BO-EADARP in Section \ref{sec::problem}, along with important notions our solution approaches will build upon.
In Section \ref{sec::Fragment-based checker algorithm}, we introduce a novel formulation of the BO-EADARP and propose an exact objective-space search algorithm to solve the BO-EADARP as well as several enhancements.
In Section \ref{sec::computational results}, we conduct extensive experiments on both bi-objective DARP and BO-EADARP benchmark instances and demonstrate the computational efficiency of the proposed algorithm. With the obtained Pareto frontiers, managerial insights are obtained for different types of service providers. The paper ends with Section \ref{sec::conclusion and extensions}, summarizing results, contributions, and outlining possible future research directions.

\section{Literature review}
\label{sec::literature}
The BO-EADARP can be formulated as a multiobjective mixed integer linear programming (MOMIP) problem. 
Exact algorithms for addressing these types of problems can be roughly divided into two classes: criterion space search algorithms  \citep[e.g.,][]{haimes1971bicriterion, chalmet1986algorithm,ralphs2006improved,boland2015criterion} that work in the space of objective function values, and decision space search algorithms \citep[e.g.][]{mavrotas1998branch,sourd2008multiobjective,vincent2013multiple, stidsen2014branch,parragh2019branch} which work in the decision space and are usually generalizations of single-objective branch-and-bound approaches. In the following, we first briefly review the literature on criterion and decision space search algorithms relevant to this work. Thereafter, we give an overview of existing algorithms for solving bi-objective DARPs. To the best or knowledge, none of them addresses the BO-EADARP.

\subsection{Criterion space search algorithms} 
Criterion space search algorithms work in the objective space. Pareto optimal solutions are found in the process of solving a sequence of single-objective problems with single-objective algorithms \citep[e.g.,][]{boland2015criterion}. One of the most popular criterion space search algorithms to handle bi-objective optimization problems is the $\epsilon$-constraint method, which was first introduced by \cite{haimes1971bicriterion}. The main idea of the $\epsilon$-constraint method is to keep one objective as the objective function and to formulate the other objective as a constraint bounded by $\epsilon$ values. In every iteration, the resulting single-objective problem is solved and the $\epsilon$ values are updated to improve the quality of the solution in terms of the second objective function. 
This method has the advantage that it is easy to implement 
while being very effective, and has therefore been widely applied in the literature to solve bi-objective integer programs  \citep[e.g.,][]{TRICOIRE20121582,parragh2009heuristic}. Another method to compute all non-dominated points is called the two-phase method. 
The two-phase method was originally introduced by \cite{ulungu1993optimisation} to solve a bi-objective assignment problem by dividing the search for non-dominated points in criterion space into two phases. The first phase of the two-phase method corresponds to finding all extreme supported efficient points. Usually, the weighted-sum algorithm (also called ``dichotomic method") of \cite{aneja1979bicriteria} is used in this stage, which works in the criterion space. After finding all extreme supported efficient solutions, the second phase aims to find other efficient solutions (including non-extreme supported efficient and non-supported efficient solutions) in the triangles defined by two consecutive extreme supported efficient solutions. More recently, \cite{boland2015criterion} introduced a generic criterion space search algorithm, called the balanced box method. It generates all non-dominated points of a bi-objective integer program and the solutions constituting them, by iteratively exploring rectangle search areas defined by the images of newly found efficient solutions in the criterion space. 
One of its advantages is that, when working on a limited time budget, a representation of the entire Pareto frontier is usually available upon termination, while, e.g., the $\epsilon$-constraint method enumerates solutions from one end of the Pareto frontier to the other and when stopped early, only solutions covering a small portion of the Pareto frontier may have been computed. 
More recently, \cite{glize2022varepsilon} proposed a tailored $\epsilon$-constraint column generation-and-enumeration algorithm. Five acceleration mechanisms are introduced to speed up the solving process of the column generation-and-enumeration algorithm for each single-objective problem. The resulting approach is shown to be more efficient than the bi-objective branch-and-price algorithm proposed by \cite{parragh2019branch} in solving the Bi-Objective Team Orienteering Problem  with Time Windows (BOTOPTW). The latter technique relies on a straightforward implementation of the pricing step. The method of \cite{parragh2019branch} was the first exact method that solved the BOTOPTW. 
It falls into the category of decision space search algorithms.

\subsection{Decision space search algorithms} 
Decision space search algorithms work in the space of decision variables. These algorithms can be regarded as generalizations of Branch-and-Bound (B\&B) algorithms in the context of MOMIP, where usually a bound set instead of a single numerical value is generated when solving a node of the branch-and-bound tree. As bounding procedures are the key ingredient of B\&B algorithms, state-of-the-art algorithms mainly focus on this element to enhance their overall efficiency. \citet{mavrotas1998branch, mavrotas2005multi} provided B\&B algorithms to find efficient solutions for multi-objective mixed 0-1 problems, where they fathom a node if the ideal point of the node problem (they work on a maximization problem) is dominated by some elements in the lower bound set that is composed of non-dominated points. 
\cite{vincent2013multiple} enhance the bounding procedure of \cite{mavrotas1998branch, mavrotas2005multi} by comparing to the bound set in each node of the B\&B tree instead of to the ideal point.
\cite{sourd2008multiobjective} define a separating hyperplane between upper and lower bound sets to judge whether a node can be fathomed in a general B\&B framework. 
Another idea for a bounding procedure in state-of-the-art B\&B algorithms \citep[e.g.][]{masin2008diversity} is to calculate a single lower bound with a surrogate objective function and decide whether the B\&B node can be fathomed with regards to the numerical value of this lower bound. 
Building on all previous results, \cite{stidsen2014branch} propose a B\&B algorithm that can compute the Pareto frontier (and one solution per non-dominated point) for a certain class of BOMIP problems, i.e., those where the Pareto frontier consists of points only. This is the case whenever the continuous variables appear in at most one objective function, like it is the case for the BO-EADARP. In the bounding procedure, \cite{stidsen2014branch} propose to partition the objective space into slices (this is done by adding constraints), then to explore each slice independently. They observe that this strategy allows more nodes to be fathomed and does not significantly increase the computational time. Also, they propose to improve the branching procedure (called ``Pareto branching") by using bounds from enumerated feasible solutions to create nodes. 
Extending the idea of Pareto branching from \cite{stidsen2014branch}, \cite{parragh2019branch} introduced a generic B\&B algorithm based on a problem-independent branching rule to address, among other benchmark problems, the BOTOPTW. To solve the BOTOPTW, they calculate the lower bound set at each node of the B\&B search tree by column generation. Several improvements derived from the integrality of objective functions are proposed to accelerate the algorithm.
In the computational experiments, the proposed B\&B algorithm is compared to several criterion space search algorithms (i.e., the algorithms of \cite{haimes1971bicriterion} and \cite{boland2015criterion}) 
and it proved to be the most efficient method among all considered algorithms in solving the BOTOPTW. Therefore, we also tailor their framework to the BO-EADARP to benchmark our approach.

\subsection{Literature on solving the bi-objective DARP}
In the context of the bi-objective DARP, a few works have been conducted to analyze the trade-off between user inconvenience and operational cost. In the work of \cite{parragh2009heuristic}, a two-phase heuristic method is developed. A set of efficient solutions is constructed, minimizing a weighted sum of total distance traveled and mean user ride time for different weight combinations, using a variable neighborhood search algorithm. It is combined with a path relinking step to obtain additional points along the approximate Pareto frontier. To validate the performance of the proposed heuristic, the $\epsilon$-constraint method is used to generate the optimal Pareto frontier on small instances. The proposed heuristic two-phase method is shown to be efficient in generating high-quality approximation. Following \cite{parragh2009heuristic}, \cite{parragh2011introducing} further considered a weighted-sum objective function consisting of the total routing cost and waiting time for all vehicles with passengers aboard. As in this paper, the objective function included a time-dependent criterion. Medium-sized instances with up to 4 vehicles and 40 customers can be solved. \cite{paquette2013combining} integrate tabu search with a reference point method, computing distances to an ideal point over all objectives. A set of supported efficient solutions is constructed by applying a weighted-sum objective function in which the search is guided by adapting weights dynamically. Total routing cost, user waiting time, and user ride time are minimized. More recently, in the work of \cite{molenbruch2017multi}, the total user ride time is set as the second objective, and a new scheduling heuristic is proposed to construct the schedule that minimizes the total user ride time. A variable neighborhood descent algorithm is designed and integrated into the multi-directional local search framework \citep{tricoire2012multi}. The problem tackled by \cite{molenbruch2017multi} is characterized by a combination of restrictions, which prevent some users from being transported together and a limitation on the set of drivers to which users may be assigned. While \citet{stallhofer2025event} evaluated several different weight settings in the context of the E-ADARP, to obtain a small set of trade-off solutions, in this work, we aim at generating the complete Pareto frontier of the BO-EADARP. 

\section{The BO-EADARP}\label{sec::problem}

In the following, we formally introduce the BO-EADARP. We then provide all necessary concepts and definitions to deal with the bi-objective nature of the problem, and we present the notion of battery-restricted fragment.

\subsection{Problem definition}
 The problem is defined on a complete directed graph $G=(V,A)$, where $V$ denotes the set of vertices and $A$ denotes the set of arcs, i.e., $A = \{(i,j):i,j \in V, i \neq j\}$. The set $V$ is composed of the set of all pickup and drop-off locations $N$, the set of recharging stations $S$, and the set of origin $O$ and destination depots $F$, i.e., $V= N \cup S \cup O \cup F$. 
 Set $N = P \cup D$, where $P =\{1,\cdots,i,\cdots,n\}$ is the set of pickup vertices (which is equivalent to the set of users) and $D =\{n+1,\cdots,n+i,\cdots,2n\}$ the set of all drop-off vertices. Each transportation request $i \in P$ consists of a vertex pair $(i,n+i)$, where $i$ is the pickup vertex and $n+i$ the corresponding drop-off vertex. 
 To ensure a certain service quality, a maximum ride time $m_i$ is also associated with each request $i$. Furthermore, a time window $[e_i,l_i]$ is defined for each node $i\in V$, where $e_i$ and $l_i$ represent the earliest and latest time at which the vehicle may start its service, respectively. A load $q_i$ and a service duration $s_i$ is also given for each node $i \in V$, where $q_i$ is positive on the pickup node $i \in P$ and negative on the corresponding drop-off node $n+i$ and we have $q_{n+i} = -q_i$. For all other nodes $j \in O \cup F \cup S$, $q_j$ and $s_j$ are equal to zero.

Each vehicle $k \in K$ must start at an origin depot $o \in O$ and end at a destination depot $f \in F$. In this paper, we stick to the original setting defined in \cite{bongiovanni2019electric}, where the number of origin depots is equal to the number of vehicles, i.e., $|O| = |K|$ while the number of destination depots might be larger (i.e., $|F| \geqslant |O|$). This means we denote by $o_k$ the origin depot of vehicle $k$. A feasible route in the BO-EADARP is defined as a path in graph $G$ starting from the origin depot and ending at a destination depot, passing through pickup, drop-off, and charging station (if required) locations, satisfying pairing, precedence, load, battery, time window, and maximum user ride time constraints. The BO-EADARP consists of designing $|K|$ routes, one for each vehicle, such that all customer nodes are visited exactly once, each recharging station and destination depot is visited at most once, and the two objectives are minimized. Vehicles are assumed to be homogeneous in terms of the number of seats available $C$ and their battery capacities $Q$.

 Let $t_{i,j}$ and $b_{i,j}$ denote the travel time and the battery consumption on each arc $(i,j) \in A$. We assume that battery discharging and recharging are proportional to travel time and charging time, respectively. The recharging rate is denoted as $\alpha$. Energy units are converted to time units by defining $h_{i,j} = b_{i,j}/\alpha$. Partial recharging is allowed. Additionally, a minimum battery level $\gamma Q$ must be respected when reaching a destination depot. The triangle inequality is assumed to hold for $t_{i,j}$ and $b_{i,j}$. 
 In order to formulate the BO-EADARP, following \citet{bongiovanni2019electric}, we use binary variables $x_{ij}^k$ which are equal to $1$ if arc $(i,j)$ is used by vehicle $k$ and $0$ otherwise. Furthermore, we use continuous variables $T_i^k$ to store the beginning of service time at a node $i$ by vehicle $k$, $L_i^k$ for the load when leaving node $i$ with vehicle $k$, $B_i^k$ to keep track of the battery level at each node $i$ and vehicle $k$, $E_s^k$ for the time available for recharging at charging station $s$ and vehicle $k$, and $R_i$ to store the excess ride time of user request $i$.

\begin{align} \label{EQsec3:objective1}
     & \min \sum\limits_{k \in K} \sum\limits_{(i,j) \in A}t_{i,j}x_{i,j}^k \\
\label{EQsec3:objective2}
    & \min \sum\limits_{i \in P}R_i 
\end{align}
subject to:
\begin{align} 
\label{EQsec3:start}
 & \sum\limits_{j \in P\cup S \cup F} x_{o^k,j}^{k} = 1, 
 && \forall k \in K
 \\
\label{EQsec3:end}
& \sum\limits_{i \in D\cup S \cup \{o^k\}}\sum\limits_{j \in F} x_{i,j}^{k} = 1, 
&& \forall k \in K
\\ 
\label{EQsec3:chargevisit}
& \sum\limits_{k \in K} \sum\limits_{i \in D\cup S \cup \{o^k\}} x_{i,j}^{k} \leqslant 1, 
&&\forall j \in F \cup S
\\ 
\label{EQsec3:inout}
& \sum\limits_{j \in V, j \neq i} x_{i,j}^{k} - \sum\limits_{j \in V, j \neq i} x_{j,i}^{k} = 0, 
&&\forall k \in K, i \in N \cup S
\\
\label{EQsec3:visitPickup}
& \sum\limits_{k \in K} \sum\limits_{i \in N, j\neq i} x_{i,j}^{k} = 1, 
&& \forall i \in P
\\
\label{EQsec3:delivery}
& \sum\limits_{j \in N, j \neq i} x_{i,j}^{k} - \sum\limits_{j \in N, j \neq n+i} x_{j,n+i}^{k} = 0, 
&& \forall k \in K, i \in P
\\ 
\label{EQsec3:precedence}
& T_i^k + s_i + t_{i,n+i} \leqslant T_{n+i}^k, 
&&\forall k \in K, i \in P
\\ 
\label{EQsec3:tw}
& e_i \leqslant T_i^k \leqslant l_i, 
&& \forall k \in K, i \in V
\\ 
\label{EQsec3:time}
& T_i^k + t_{i,j} + s_i - M_{i,j}(1-x_{i,j}^k) \leqslant T_j^k, 
&& \forall k \in K, (i,j) \in A 
\\ 
\label{EQsec3:ride}
& T_{n+i}^k - T_{i}^k - s_i \leqslant m_i, 
&& \forall k \in K, i \in P
\\ 
\label{EQsec3:excess}
& R_i \geqslant T_{n+i}^k - T_{i}^k - s_i - t_{i,n+i}, 
&& \forall k \in K, i \in P
\\ 
\label{EQsec3:load1}
& L_i^k + q_j  - G_{i,j}^k(1-x_{i,j}^k) \leqslant L_j^k, 
&& \forall k \in K, (i,j) \in A 
\\ 
\label{EQsec3:load2}
& L_i^k + q_j + G_{i,j}^k(1-x_{i,j}^k) \geqslant L_j^k, 
&& \forall k \in K, (i,j) \in A 
\\ 
\label{EQsec3:loadLB}
& L_i^k \geqslant \max\{0,q_i\}, 
&& \forall k \in K, i \in N
\\ 
\label{EQsec3:loadUB}
& L_i^k \leqslant \min\{C,C+q_i\}, 
&& \forall k \in K, i \in N
\\ 
\label{EQsec3:loadZero}
& L_i^k = 0, 
&& \forall k \in K, i \in \{o^k\} \cup F \cup S
\\ 
\label{EQsec3:SoCZero}
& B_{o^k}^k = B_0^k, 
&& \forall k \in K
\\ 
\label{EQsec3:SoC1}
& B_j^k \leqslant B_i^k - b_{i,j} + Q(1-x_{i,j}^k), 
&& \forall k \in K, i \in V \setminus S, j \in V  \setminus \{o^k\}, i \neq j
\\ 
\label{EQsec3:SoC2}
& B_j^k \geqslant B_i^k - b_{i,j} - Q(1-x_{i,j}^k), 
&& \forall k \in K, i \in V \setminus S, j \in V  \setminus \{o^k\}, i \neq j
\\
\label{EQsec3:SoC3}
& B_j^k \leqslant B_s^k + \alpha E_s^k - b_{s,j} + Q(1-x_{s,j}^k), 
&& \forall k \in K, s \in S, j \in P \cup F \cup S, s \neq j
\\ 
\label{EQsec3:SoC4}
& B_j^k \geqslant B_s^k + \alpha E_s^k - b_{s,j} - Q(1-x_{s,j}^k), 
&& \forall k \in K, s \in S, j \in P \cup F \cup S, s \neq j
\\ 
\label{EQsec3:SoC5}
& Q \geqslant B_s^k + \alpha E_s^k, 
&& \forall k \in K, s \in S
\\ 
\label{EQsec3:SoC6}
& B_i^k \geqslant \gamma Q, 
&& \forall k \in K, i \in F
%
\\
\label{EQsec3:timeCharge1}
 &   E_s^k \leqslant T_i^k - t_{s,i} - T_s^k + \tilde M_{s,i}^k\left(1-x_{s,i}^{k}\right), 
 && \forall s \in S, i \in P\cup{S}\cup{F}, k\in K, i \ne s
\\ 
\label{EQsec3:timeCharge2}
&    E_s^k \geqslant T_i^k - t_{s,i} - T_s^k - \tilde M_{s,i}^k\left(1-x_{s,i}^{k}\right), 
&& \forall s \in S, i \in P\cup{S}\cup{F}, k\in K, i \ne s
\\ 
\label{sec3: cons26}
&    x_{i,j}^k \in \{0,1\}, 
&& \forall k \in K, (i,j) \in A 
\\ 
\label{sec3: cons27}
&    B_i^k \geqslant 0, 
&& \forall k \in K, i \in V
\\ 
\label{sec3: cons28}
&    E_s^k \geqslant 0, 
&& \forall k \in K, s \in S
\end{align}

Objective function~\eqref{EQsec3:objective1} minimizes the total travel time, and objective function~\eqref{EQsec3:objective2} minimizes the total excess user ride time. Constraints~\eqref{EQsec3:start} and ~\eqref{EQsec3:end} make sure that each vehicle leaves from its depot and returns to one of the end depots. Constraints~\eqref{EQsec3:chargevisit} guarantee that each charging station is visited at most once. Flow conservation is taken care of by~\eqref{EQsec3:inout}.   Constraints~\eqref{EQsec3:visitPickup}--\eqref{EQsec3:precedence} guarantee that each pickup node is visited exactly once, that the same vehicle that visits pickup node $i$ also visits drop-off node $n+i$, and in the correct order. 
Constraints~\eqref{EQsec3:tw} and \eqref{EQsec3:time} take care of time windows and of correctly setting the beginning of service time variables. 
Maximum user ride time restrictions and the correct computation of the excess user ride times are handled by \eqref{EQsec3:ride} and \eqref{EQsec3:excess}. Constraints \eqref{EQsec3:load1} -- \eqref{EQsec3:loadZero} make sure that the vehicle capacity (seats available) is respected at all times. The remaining constraints handle battery-related restrictions. This requires the correct update of the state-of-charge variables (SoC) at pickup and drop-off nodes (see constraints \eqref{EQsec3:SoC1} and  \eqref{EQsec3:SoC2}) and after visiting a charging station (see constraints \eqref{EQsec3:SoC3} and  \eqref{EQsec3:SoC4}), where the latter work in combination with constraints \eqref{EQsec3:timeCharge1} and \eqref{EQsec3:timeCharge2}, which compute the time available to recharge during a visit at a charging station. Finally, constraints \eqref{EQsec3:SoC5} and \eqref{EQsec3:SoC6} make sure that the battery capacity and the minimum battery level at the depot are respected.

Tight big-M values $M_{i,j} = l_i + t_{ij} + s_i$ for each arc $(i,j) \in A$ are defined to make sure that \eqref{EQsec3:time} are only binding when $x_{ij}^k = 1$.
Similarly, we take $G_{i,j}^k = \max\{C^k,C^k+l_i\}$ for ~\eqref{EQsec3:load1} and ~\eqref{EQsec3:load2} and $\tilde M_{s,i}^k = T_p$ for ~\eqref{EQsec3:timeCharge1} and ~\eqref{EQsec3:timeCharge2}. 

\subsection{Preliminaries in bi-objective optimization} \label{sec::preliminary in bi-objective context}
In order to deal with the bi-objective nature of the problem, we resort to the so-called Pareto approach. This means that we concurrently optimize both objectives, and we aim at generating the Pareto frontier and the solutions constituting this frontier.
Let $\mathcal{X}$ denote the set of all feasible solutions and $z(x)$ the mapping of a solution $x \in \mathcal{X}$ into the criterion space. 
We now more formally introduce all necessary concepts, relying on the definitions of 
\cite{przybylski2008two}.


\begin{defn}[efficient solution]\label{BO defn2}
A feasible solution $x^* \in \mathcal{X}$ is called \textit{efficient (or Pareto optimal)} if there does not exist any other feasible solution $x \in X$ such that $z_k(x) \leqslant z_p(x^*), k = 1,2$, with at least one strict inequality. $z(x^*)$ is called a \textit{non-dominated} point. The set of efficient solutions is denoted as $\mathcal{X}_E$, and $\mathcal{Y}_N := \{z(x): x \in \mathcal{X}_E\}$ is called the Pareto frontier or efficient frontier.
\end{defn}

After defining efficient solutions $\mathcal{X}_E$, we partition $\mathcal{X}_E$ into the following two sets:
\begin{defn}[supported efficient solutions]
    A feasible solution $x$ is called a \textit{supported efficient solution} if $x$ is an optimal solution of the weighted-sum problem:
    \begin{equation}
    \displaystyle{\min_{x \in \mathcal{X}} \lambda z_1(x) + (1-\lambda)z_2(x)}\\
    \end{equation}
    where $0 < \lambda < 1$.
\end{defn}

\begin{defn}[non-supported efficient solutions]
    A feasible solution $x$ is called a \textit{non-supported efficient solution} if $x$ is not an optimal solution to a weighted-sum problem under any weight combination.
\end{defn}

We use $\mathcal{X}_{SE}$ to denote the set of supported efficient solutions, and the image of $\mathcal{X}_{SE}$ in the objective space is called the set of \textit{supported non-dominated points}, denoted as $\mathcal{Y}_{SN}$. The set of non-supported efficient solutions is denoted as $\mathcal{X}_{NE}$ and its image in the objective space is called the set of \textit{non-supported non-dominated points} and is denoted as $\mathcal{Y}_{NN}$. 


\begin{observation}
All non-supported non-dominated points are located in the interior of $conv(Y+\mathbb{R}_{\geq}^p)$, where $p$ gives the number of objectives, $p=2$ in our case. All supported non-dominated points are points of $conv(Y+\mathbb{R}_{\geq}^p)$, where $\mathbb{R}_{\geq}^p := \{y \in \mathbb{R}^p: y \geq 0\}$ is the non-negative orthant of $\mathbb{R}^p$ 
and $conv(\mathcal{Y})$ is the convex hull of feasible set $\mathcal{Y}$ in the objective space.
\end{observation}




\subsection{Battery-restricted fragment} \label{sec::fragment}
As proposed by \cite{su2023deterministic}, we rely on the notion of battery-restricted fragments (Definition \ref{fragment}).
A battery-restricted fragment has two key properties:
\begin{enumerate}
    \item Each E-ADARP route can be represented as a concatenation of BRFs, optional recharging stations, and the origin and destination depots.
    \item The minimum excess user ride time of a feasible E-ADARP solution depends solely on the minimum excess user ride time of each BRF included in the route (as shown in Theorem 1 of \cite{su2023deterministic}).
\end{enumerate}

\begin{defn}[Battery-restricted fragment, \cite{su2023deterministic}] \label{fragment}
Assuming that $frag = (i_1,i_2, \cdots,i_k)$ is a sequence of pickup and drop-off nodes, where the vehicle arrives empty at $i_1$ and leaves empty at $i_k$ and has passenger(s) on board at other nodes. Then, we call ``$frag$" a \BRFrag if there exists a feasible route of the form:
$$(o,s_{i_1},\cdots,s_{i_v},\overbrace{i_1,i_2, \cdots,i_k}^{frag},s_{i_{v+1}},\cdots,s_{i_m},f),$$ where $s_{i_1},\cdots,s_{i_v},s_{i_{v+1}},\cdots,s_{i_m} (v,m \geqslant 0)$ are recharging stations, and $o \in O$, $f \in F$.
\end{defn}

It should be noted that the route defined in Definition \ref{fragment} may not contain a recharging station. In this case, we have $v=m=0$ and the \BRFrag is equivalent to a fragment defined in \cite{rist2021new}.
Let $i^+$ denote the pickup of user $i$ and $i^-$ its corresponding drop-off. Figure \ref{battery-restricted fragment example} presents an example of a feasible route that consists of two battery-restricted fragments, i.e., $frag_1 = \{1+,2+,1-,2-\}$ and $frag_2 = \{3+,3-\}$. Note that $frag_1 \cup frag_2$ is not a battery-restricted fragment, as the vehicle becomes empty (load = 0) at intermediate node 2-.
Based on Definition \ref{fragment}, it is clear that each E-ADARP route can be regarded as the concatenation of several battery-restricted fragments, recharging stations (if required), an origin depot, and a destination depot.

\begin{figure}[H]
\centering
\includegraphics[width=16cm]{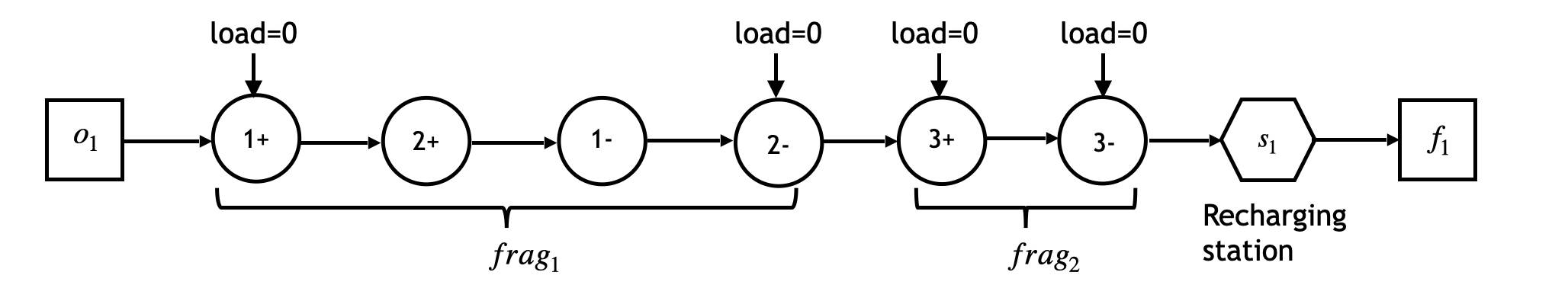}
\caption{\centering Example of \BRFrags}
\label{battery-restricted fragment example}
\end{figure}

\section{Solution approach}
\label{sec::Fragment-based checker algorithm}
In this section, we present a novel criterion-space search algorithm that we call \textit{\SOSfull}, developed to efficiently solve the BO-EADARP by exploiting a key structural unit, namely the Battery-Restricted Fragment (BRF), as introduced in Section \ref{sec::fragment}.

The \SOSfull algorithm is an exact method (see Theorem \ref{exactness of algorithm} at the end of this section) and can be interpreted as a criterion-space search technique that partitions the objective space into distinct \textit{Service Quality Range Regions}. At the core of this approach lies the \textit{\FBCfull} algorithm, which constructs feasible solutions by ``smartly" selecting and concatenating BRFs that satisfy a given service quality range constraint.

We begin by introducing the \FBCfull algorithm in Section \ref{DARP two-stage}, followed by a detailed presentation of the \SOSfull algorithm in Section \ref{sec:SOS}.



\subsection{\FBCfull Algorithm} \label{DARP two-stage}

The \FBCfull algorithm (hereafter, \FBC) is designed to efficiently find a feasible solution within a region of the objective space, referred to as the Service-Quality Range Region \( R(z^t, z^b) \). We named it as ``Service-Quality Range Region" as it contains all points whose service-quality levels lie between $[y_2^b,y_2^t)$. Its definition is given below:

\begin{defn}[Service Quality Range Region]
    Let \( z^t = (y_1^t, y_2^t) \) and \( z^b = (y_1^b, y_2^b) \) be two points in the objective space. The service quality range region \( R(z^t, z^b) \) is defined as:
    \[
    R(z^t, z^b) = \{(z_1, z_2) \mid 0 \leq z_1 < y_1^b, \, y_2^b \leq z_2 < y_1^t \}.
    \]
    This region is illustrated in Figure~\ref{given objective space}.
\end{defn}

\begin{figure}[t]
\centering
\includegraphics[width=6cm]{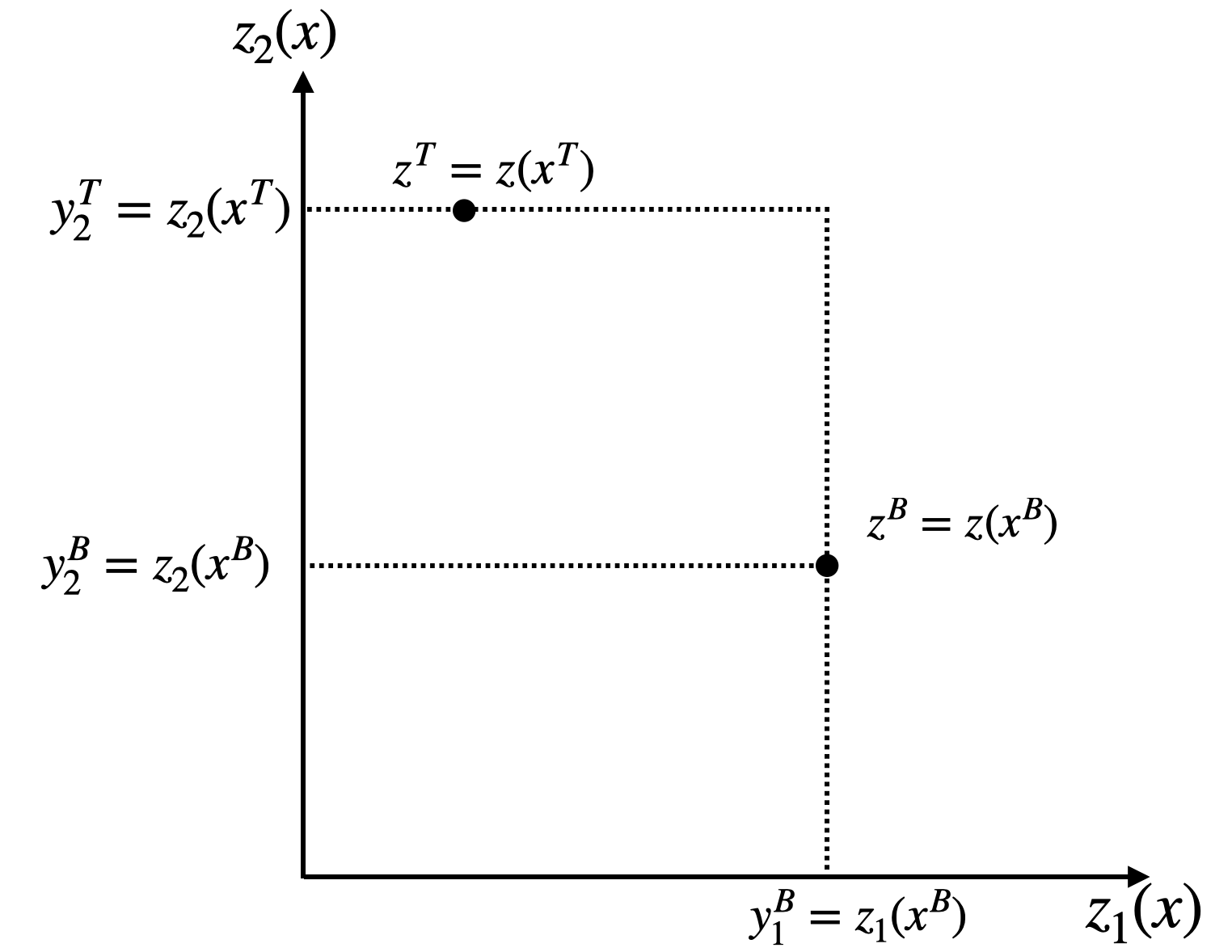}
\caption{\centering Example of a service-quality range}
\label{given objective space}
\end{figure}

Given a region \( R(z^t, z^b) \), the \FBC algorithm identifies a corresponding feasible solution by iteratively executing the following two steps:

\begin{itemize}
    \item \textbf{Step 1: Select compatible BRFs} Here, we select a set of BRFs that together cover all customer requests exactly once within the service quality range.  
This can be done by solving the following knapsack problem to ``put" a set of fragments that can make up a feasible solution into the ``knapsack":

\begin{equation} \label{1.1}
\min \sum\limits_{f \in \mathcal{F}} T_f x_f
\end{equation}
subject to:
\begin{align} \label{1.2}
    \sum \limits_{f \in \mathcal{F}} a_{i,f}x_f & = 1
    && \forall i \in P
    \\
%
\label{1.3}
    \sum \limits_{f \in \mathcal{F}} R_f x_f & < y_2^t
    \\
%
\label{1.4}
    \sum \limits_{f \in \mathcal{F}} R_f x_f & \geq y_2^b
    \\
%
\label{cuts on fragments}
    \sum \limits_{f \in \xi} x_f & \leq |\xi| - 1 
    && \xi \in \Xi
    \\
%
\label{1.6}
    x_f & \in \{0,1\}
    && \forall f \in \mathcal{F}
\end{align}

In this model, $x_f$ is a binary variable that is equal to one if fragment $f$ is selected and zero otherwise. We denote $\mathcal{F}$ as the set of all feasible fragments, which are generated by applying the fragment enumeration process in \cite{su2023deterministic,su2024branch}. The objective is to minimize the sum of total travel time on selected fragments, where $T_f$ denotes the travel time on fragment $f$. We use $R_f$ to denote the excess user ride time of fragment $f$. In Constraints (\ref{1.2}), we make sure that the selected fragments cover all the requests, and each request is visited exactly once. $a_{i,f}$ is the coefficient which is equal to one if request $i$ is included in fragment $f$ and zero otherwise. Constraints (\ref{1.3}) and (\ref{1.4}) ensure that the total excess user ride time of selected fragments is within the service quality range $[y_2^b,y_2^t)$. Constraints (\ref{cuts on fragments}) are called \textit{``incompatible fragments cuts''}, which build the link between steps 1 and 2. In case that the fragments $f$ in the combination $\xi$ cannot compose a feasible solution or the composed solution is not a new non-dominated candidate (checked in step~2 as described below), a cut is added to avoid selecting the same combination of fragments. $\Xi$ is the set of all incompatible fragment sets found until the current iteration. Finally, Constraints (\ref{1.6}) are the domain constraints for binary variable $x_f$.

\item \textbf{Step 2: Check Feasibility}  After selecting fragments, we need to check whether there exists an E-ADARP solution composed of these fragments that satisfies all constraints (battery, charging, and time windows) while being included in the region $R(z^t, z^b)$. As fragments have been selected in step 1, we can fix all \( x_{i,j} = 1 \) for the arcs within the selected fragments. Then, we solve the original MILP model introduced in Section~\ref{sec::problem}. There are two cases when constructing a solution with fragments:
\begin{itemize}
    \item If a solution \( x \) exists and $z_1(x)< y_1^b$, then a desired feasible solution is constructed using the fragment set generated in \textit{Step 1}.
    \item Otherwise, we rerun \textit{Step 1} with an incompatible fragment cut (i.e., Constraint~\eqref{cuts on fragments}) added to the knapsack problem.
\end{itemize}


It should be noted that in the first case if we can find a solution $x$ using selected fragments in step 1 and $x$ satisfies the condition $z_1(x) < y_1^b$, this solution is actually a candidate for a new efficient solution. The reason is that this solution has $z_2(x) \in [y_2^b,y_2^t)$ and $z_1(x) < y_1^b$.
To improve efficiency, we follow the approach in \cite{su2024branch} 
for
constructing a new graph \( G_{sp} \). On this graph, each fragment is abstracted as a single edge connecting two nodes. This transformation reduces computational time and still guarantees optimality, as proven in Theorem~2 of \cite{su2024branch}.

\begin{exmp}[Transform a fragment to an edge]
    In Figure~\ref{frag-edge representation}, the first route is composed of a fragment $\mathcal{F}$ which includes four nodes (i.e., $\mathcal{F} = \{1+,2+,1-,2-\}$), each associated with a time window. This fragment can be abstracted to an edge between node $1+$ and node $2-$ as shown in the second route.
    \begin{figure}[t]
    \centering
    \includegraphics[width=11cm]{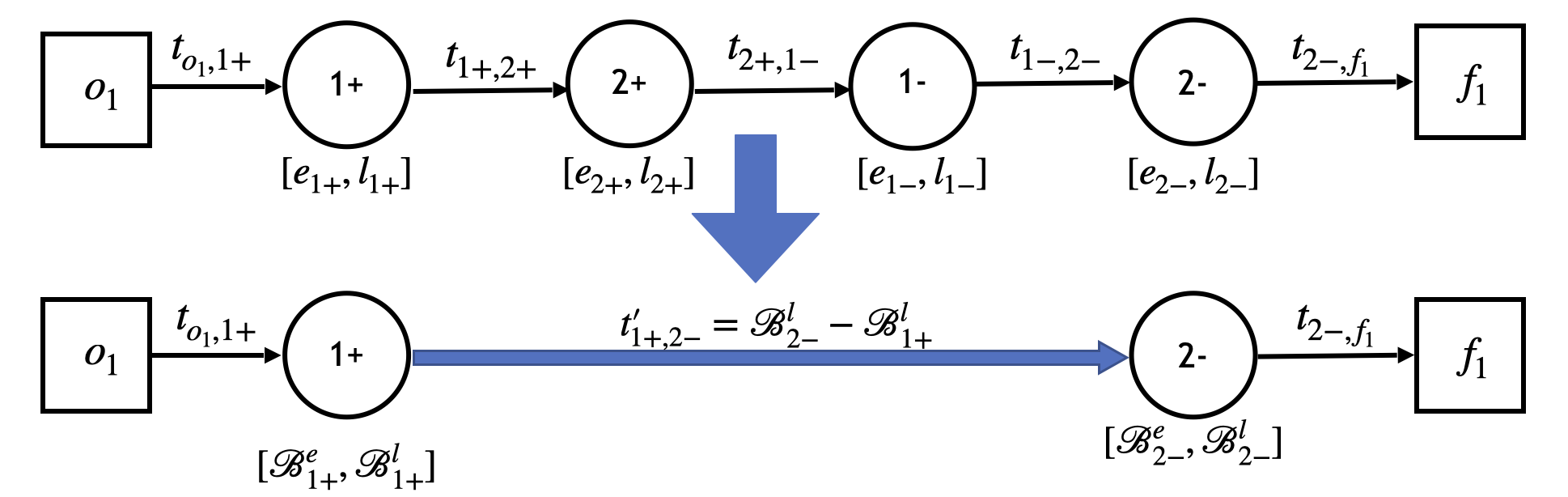}
    \caption{\centering Example of transforming fragment to edge representation}
    \label{frag-edge representation}
    \end{figure}
\end{exmp}

The detailed formulation is presented in \ref{appendix: stage 2 formulation}.

\end{itemize}

\begin{theorem}\label{theo::FBC}
If a feasible solution exists within the region \( R(z^t, z^b) \), then the \FBC algorithm is guaranteed to find a feasible solution in \( R(z^t, z^b) \).
\end{theorem}

\begin{proof}
Given a range \( R(z^t, z^b) \), 
there is only a finite number of fragment sets that is generated in the first step of the \FBC algorithm. These include those corresponding to all the feasible fragments. Since the fragment set associated with at least one feasible solution is included (if a feasible solution exists), step 2 of the algorithm must be able to construct a feasible solution using this set. Therefore, the \FBC algorithm is guaranteed to identify a feasible solution if one exists in the specified range.
\end{proof}

Notably, both steps of the \FBC algorithm are computationally efficient. The first step reduces to a knapsack problem, while in the second step, most route segments are already determined by the selected fragments. This structure enables our method to achieve substantial computational savings.

However, step 1 may generate many fragment sets that do not lead to feasible solutions or not fall in the desired region, and thus requiring multiple iterations to identify a valid one. To mitigate this issue, we incorporate enhancements into the computation process to avoid selecting unpromising fragment sets. The details of these improvements are presented in Section~\ref{sec::enhancement on algorithm}.

\subsection{\SOSfull Algorithm}\label{sec:SOS}

\begin{algorithm}[htp]
\caption{\SOSfull Algorithm}
\label{alg:sos_algorithm}
\begin{algorithmic}[1]
\Require BO-EADARP problem instance
\Ensure Set of efficient solutions $\mathcal{X}_E$ and non-dominates points $\mathcal{Y}_N$

\State  Solve two objectives in lexicographical order (i,e., $lex\min_{x\in \mathcal{X}}\{z_1(x), z_2(x)\}$ and $lex\min_{x\in \mathcal{X}}\{z_2(x), z_1(x)\}$) to compute efficient solutions $x^T$ and $x^B$, respectively
\State Initialize $z^T \leftarrow (z_1(x^T), z_2(x^T))$ and $z^B \leftarrow (z_1(x^B), z_2(x^B))$
\State $\mathit{Rec} \leftarrow [R(z^T, z^B)]$ \Comment{Initialize rectangle list}
\State $\mathcal{X}_E \leftarrow \{x^T, x^B\}$ \Comment{Initialize efficient solution set}

\While{$\mathit{Rec} \neq \emptyset$}
    \State $R(z^t, z^b) \leftarrow \mathit{Rec}.\text{pop}()$ \Comment{Extract and remove last rectangle}
    
    \If{new solution $x^{new} \leftarrow \text{\FBC}(R(z^t, z^b))$ is obtained}
        \State $\mathcal{X}_E \leftarrow \mathcal{X}_E \cup \{x\}$ \Comment{Add solution to efficient set}
        \State $z^{new} \leftarrow (z_1(x^{new}), z_2(x^{new}))$ 
        
        \If{$z^t = z^T$} \Comment{Special case: processing top rectangle}
            \State Add $R(z^T, z^{new})$ and $R(z^{new}, z^b)$ to $\mathit{Rec}$
            \State \textbf{continue} \Comment{Leave and start the next while iteration}
        \EndIf
        
        \For{each $R(\bar{z}^t, \bar{z}^b) \in \mathit{Rec}$ from end to beginning} \Comment{General case: Update $Rec$ according to $z^{new}$}
            \If{$z^{new}$ dominates $\bar{z}^t$} 
                \State Remove $R(\bar{z}^t, \bar{z}^b)$ from $\mathit{Rec}$
            \ElsIf{$z^{new}$ dominates $\bar{z}^b$} 
                \State Replace $R(\bar{z}^t, \bar{z}^b)$ with $R(\bar{z}^t, z^{new})$ in $\mathit{Rec}$
                \State Add $R(z^{new}, z^b)$ to end of $\mathit{Rec}$
                \State \textbf{break} \Comment{Leave the for loop}
            \Else
                \State Add $R(\bar{z}^t, z^{new})$ and $R(z^{new}, z^b)$ to the end of $\mathit{Rec}$ 
                \State \textbf{break} \Comment{Leave the for loop}
            \EndIf
   
        \EndFor

    \EndIf
\EndWhile

\State Filter dominated solutions from $\mathcal{X}_E$ and compute the corresponding non-dominated points $\mathcal{Y}_N$
\State \Return $\mathcal{X}_E$, $\mathcal{Y}_N$
\end{algorithmic}
\end{algorithm}

The pseudo-code for the \SOS algorithm is presented in Algorithm \ref{alg:sos_algorithm}. Given an instance of the BO-EADARP problem, the algorithm outputs the set of efficient solutions and the corresponding non-dominated points, denoted as $\mathcal{X}_E$, $\mathcal{Y}_N$, respectively.  

The algorithm begins by initializing the first search region $R(z^T, z^B)$ through solving the two objectives of the BO-EADARP in lexicographical order (i,e., $lex\min_{x\in \mathcal{X}}\{z_1(x), z_2(x)\}$ and $lex\min_{x\in \mathcal{X}}\{z_2(x), z_1(x)\}$). 
The term $lex\min_{x\in \mathcal{X}}\{z_1(x), z_2(x)\}$ describes the process in which we find solutions with the smallest values for $z_2(x)$ among all feasible solutions with the smallest values for $z_1(x)$ and similar for $lex\min_{x\in \mathcal{X}}\{z_2(x), z_1(x)\}$. The obtained non-dominated points $z^T$ and $z^B$ define the first search area containing all other (yet to be identified) non-dominated points.
This initialization provides two supported efficient solutions $x^T$ and $x^B$, and their corresponding supported non-dominated points $z^T$ and $z^B$. 
Subsequently, the algorithm enters an iterative search loop (lines 5–27), which continues until the list of search regions, denoted by $Rec$, becomes empty. Specifically, in each iteration, the algorithm takes the last region from $Rec$ (line 6) and applies the \FBC algorithm to generate a new candidate solution within the selected region (lines 7–8). Based on the properties of the obtained solution, the algorithm updates the search region list $Rec$ according to the cases outlined in lines 10–25. Finally,   
upon termination of the while loop, the algorithm (line 28) filters the collected candidate solutions to identify and return the final set of efficient solutions and the corresponding non-dominated points. To clarify the search procedure, we provide an illustrative example.

\begin{exmp}
Here, we provide a simple example to illustrate the progression of Algorithm \ref{alg:sos_algorithm}.

Initially, in lines 1-4 of Algorithm \ref{alg:sos_algorithm}, we obtain two lex-min solutions $x^T$ and $x^B$  as illustrated in Figure \ref{SOS1}, and define the initial region list $ Rec = [R(z^T, z^B)]$.
Subsequently, we enter the while loop spanning lines 5-27:
\begin{enumerate}
    \item In the first iteration, the region $R(z^t, z^b) = R(z^T, z^B)$ is dequeued. 
    
    We apply the \FBC algorithm to $R(z^t, z^b)$ in line 5.  Assume the generated efficient solution candidate is $\overline{x}^3$ with the corresponding objective space point $\bar{z}^3$ (marked by the red point in Figure \ref{SOS1}). Since the extracted region is the top region of the form $R(z^T, *)$, the algorithm falls into the case defined by lines 10-13. Therefore, we have new $Rec$ as illustrated in Figure \ref{SOS2}:
    \[
    Rec = [R(z^T, \bar{z}^3),\quad R(\bar{z}^3, z^B)].
    \]

    \item In the second iteration, the region $R(z^3, \bar{z}^B)$ is dequeued, and we have $Rec=[R(z^T, \bar{z}^3)]$.

    Assuming the \FBC algorithm generates a new solution $\overline{x}^4$ with the corresponding objective space point $\overline{z}_4$ (see Figure \ref{SOS3}). In this case, we have $\bar{z}^4$ dominates  $\bar{z}^3$ but not $z^T$. As a result, the algorithm  falls into the case defined by lines 17-20. Hence, we replace $R(z^T, \bar{z}^3)$ by $R(z^T, \bar{z}^4)$ and obtain the new $Rec$ (see Figure \ref{SOS4}) as follows:
    \[
    Rec = [R(z^T, \bar{z}^4),\quad R(\bar{z}^4, z^B)].
    \]

    \item Assuming there are no other feasible solutions in $R(z^T, \bar{z}^4), R(\bar{z}^4, z^B)$, then the \SOS algorithm sequentially dequeues these three regions without finding new solutions, leading to termination. The final efficient solution set is $\{x^T, \overline{x}^4, x^B\}$ with the corresponding non-dominated points $\{z^T, \overline{z}^4, z^B\}$.
\end{enumerate}

    \begin{figure}[t]
    \centering
    \subfigure[First candidate solution.]{
    \includegraphics[width=8cm]{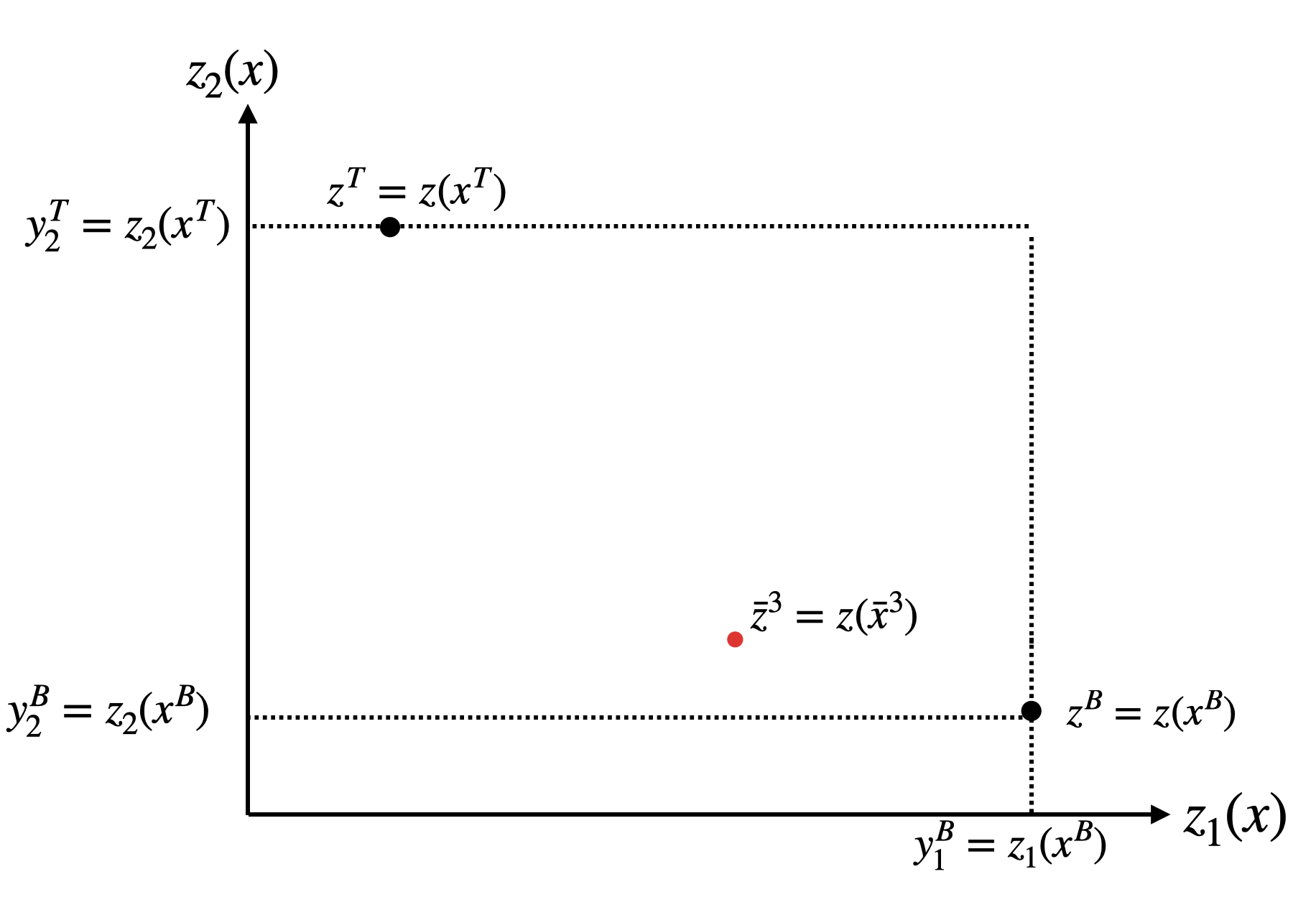}
    \label{SOS1}
    }
    \hfill
    \subfigure[$Rec$ after first loop.]{
    \includegraphics[width=8cm]{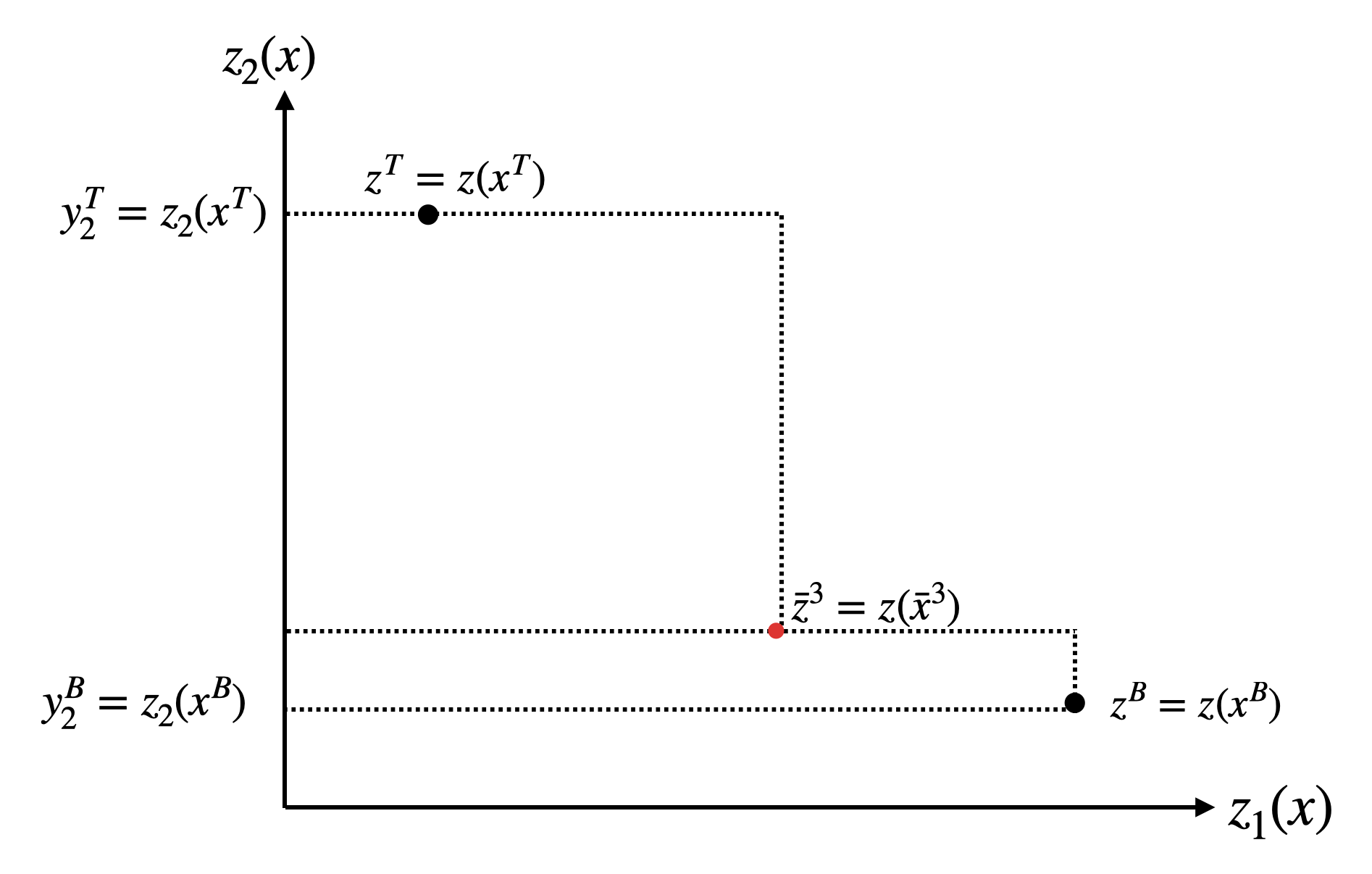}
    \label{SOS2}
    }
    \hfill
    \subfigure[Second candidate solution.]{
    \includegraphics[width=8cm]{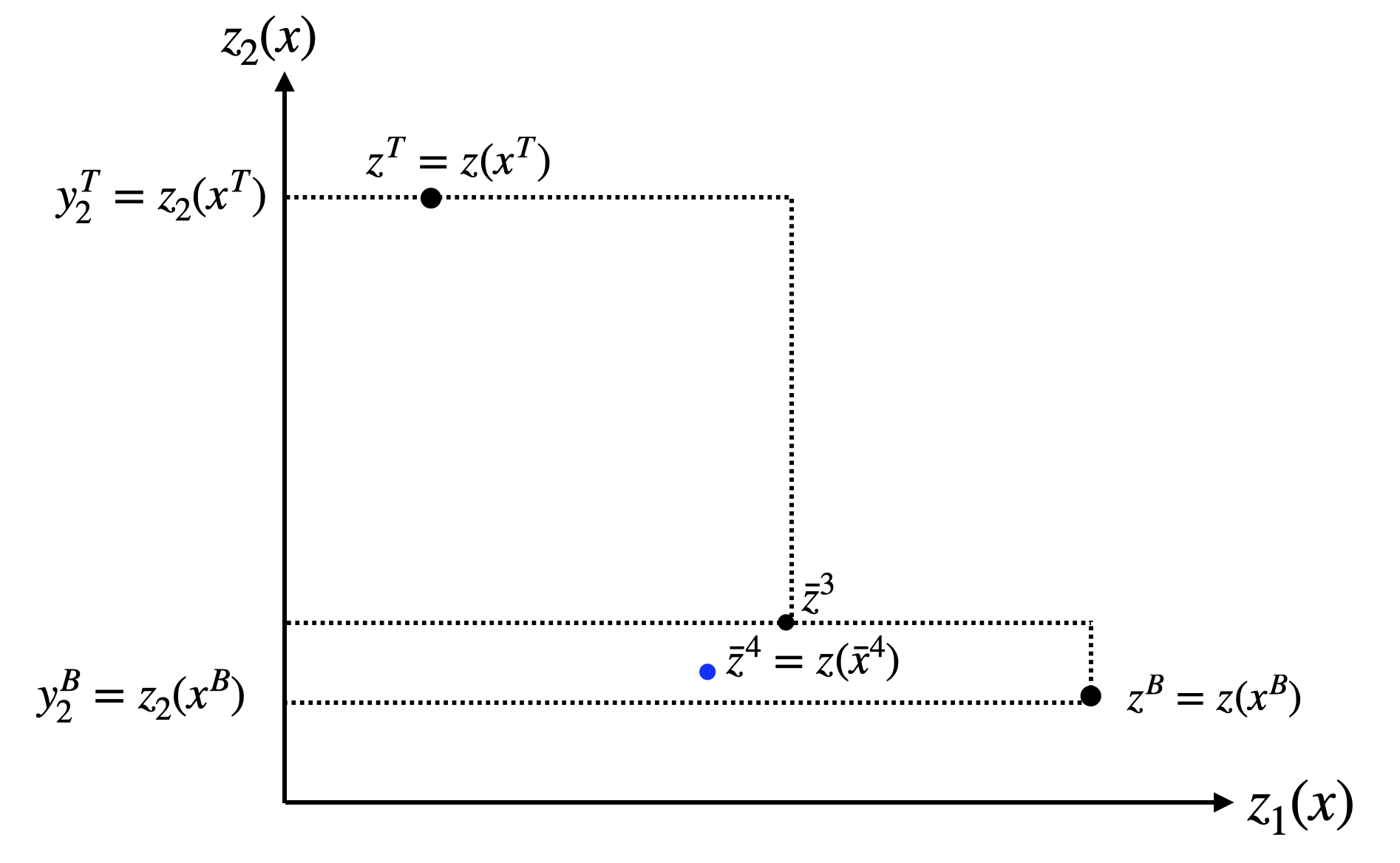}
    \label{SOS3}
    }
    \hfill
    \subfigure[$Rec$ after second loop.]{
    \includegraphics[width=8cm]{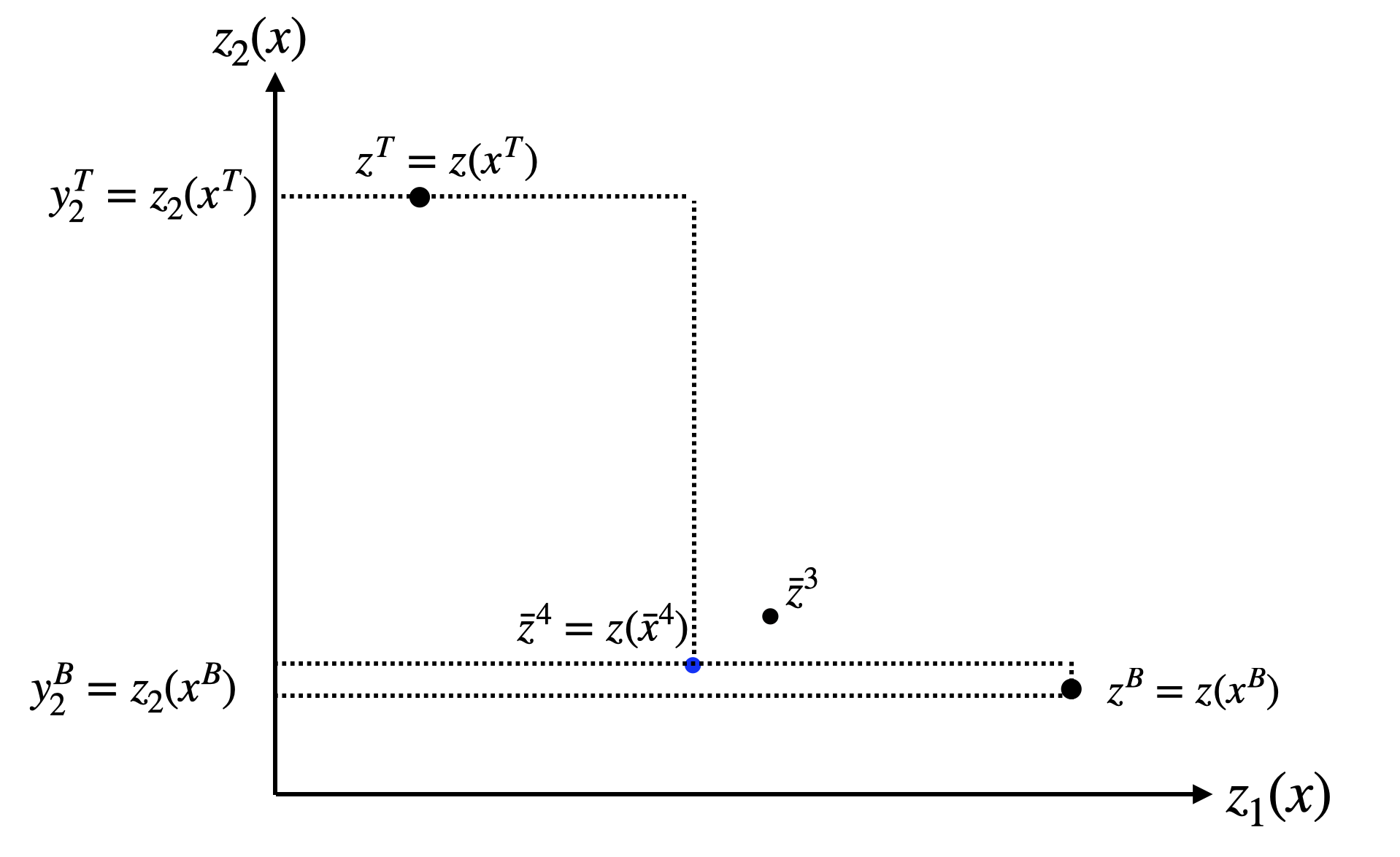}
    \label{SOS4}
    }
\caption{\centering Illustration for \SOS algorithm}
\end{figure}
\end{exmp}



\begin{figure}[htp]
\centering
\includegraphics[width=14cm]{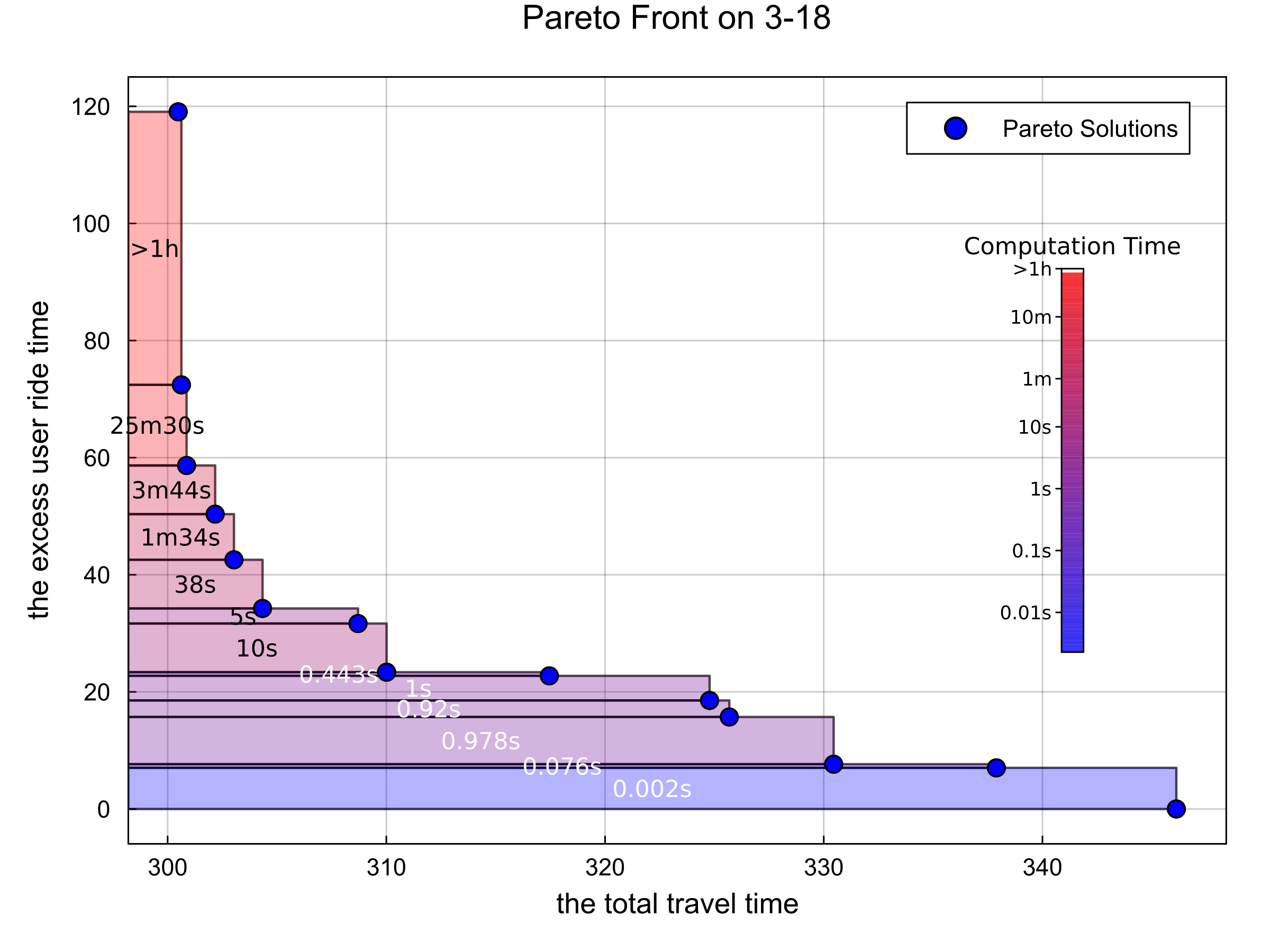}
\caption{\centering Computational time on different rectangles for instance a3-18}
\label{CPU on rectangles}
\end{figure}

It should be noted that the proposed algorithm can also start with $R(z^t,z^b)$ where $z^t$ and $z^b$ are not necessarily lex-min solutions but the defined rectangle area should be sufficiently large to include all non-dominated points. In this case, $\mathcal{X}_E$ is initialized with an empty set.

\paragraph{Analysis}

The \SOS algorithm is an exact method (see Theorem \ref{exactness of algorithm} at the end of this section) that can be interpreted as a criterion-space search strategy. The algorithm maintains a list of search regions, denoted by \( Rec \), each represented as a rectangle \( R(z^t, z^b) \) in the objective space. These rectangles are ordered in decreasing fashion along the \( z_2 \)-axis, meaning the algorithm explores the search space from bottom to top. This search direction offers several key advantages:

\begin{itemize}
    \item \textit{Higher efficiency in low \( z_2 \) regions:} In the lower regions of the objective space which correspond to smaller \( z_2 \) values, there are typically fewer feasible fragment combinations. As a result, the \SOS algorithm tends to be more efficient when applied to these regions.
    
    \item \textit{Early restriction of the criterion space:} Solutions found in lower rectangles can be used to prune infeasible or dominated portions of the search space in higher rectangles. For example, as shown in Figure \ref{SOS3}, solution \( \overline{z}_4 \) allows us to eliminate part of the region \( R(z^T, \overline{z}^3) \), thus reducing redundant computations.
\end{itemize}

Despite these strengths, the algorithm's performance tends to degrade in upper rectangles with higher \( z_2 \) values. These regions are associated with higher total excess user ride time and often involve a significantly larger number of feasible fragment combinations, involving more iterations of \FBC. For example, Figure \ref{CPU on rectangles} illustrates the computational time of the \SOS algorithm at each rectangle for an instance with 3 vehicles and 18 requests. It is clear that the top-left rectangle requires the longest computation time, as it contains the highest number of feasible fragment combinations, while the bottom-right rectangle only consumes 0.002 seconds to enumerate all feasible fragment combinations. This suggests that the proposed method is particularly well-suited for scenarios in which service quality (i.e., lower excess ride time) is a primary concern.

\begin{theorem}[Exactness] \label{exactness of algorithm}
The \SOS algorithm presented in Algorithm \ref{alg:sos_algorithm} is guaranteed to find all non-dominated points.
\end{theorem}

\begin{proof}
We prove the theorem by contradiction. Suppose, for the sake of contradiction, that there exists a non-dominated point \( z^* \) that is not identified by the algorithm. Let \( R(z^t, z^b) \) be the last rectangle containing \( z^* \) that is extracted at line 6 of Algorithm \ref{alg:sos_algorithm}. 

According to Theorem \ref{theo::FBC}, when the \FBC algorithm is applied to the region \( R(z^t, z^b) \), it must generate a feasible solution $x'$ with the corresponding objective space point \( z' \neq z^* \). We consider the following two cases:

\begin{enumerate}
    \item If \( z_2(z') > z_2(z^*) \), then the new rectangle \( R(z', z^b) \) contains \( z^* \).
    
    \item If \( z_2(z') \leq z_2(z^*) \), then, since \( z^* \) is a non-dominated point, it must hold that \( z_1(z') > z_1(z^*) \). In this case, the rectangle \( R(z^t, z') \) contains \( z^* \).
\end{enumerate}

By construction in Algorithm~\ref{alg:sos_algorithm}, the rectangle \( R(z', z^b) \) is always added to the list \( Rec\) in lines 10--25. In addition, a second rectangle is also added: either \( R(\bar{z}^t, z') \), which contains \( R(z^t, z') \), if line 18 or 22 is executed; or \( R(z^t, z') \) otherwise. 
 
 Consequently, one of these new rectangles, which still contains \( z^* \), will eventually be extracted again in a subsequent iteration (line 6). 
This contradicts our assumption that \( R(z^t, z^b) \) was the last rectangle containing \( z^* \) to be extracted. Therefore, the initial assumption must be false, and the algorithm must find all non-dominated points.
\end{proof}

\subsection{Enhancements} \label{sec::enhancement on algorithm}
In this section, we introduce several enhancements that aim at improving the efficiency of our \SOS algorithm. These enhancements focus on the following aspects:

\begin{itemize}
    \item \textit{Enhancement of the \FBC algorithm:} Refinement of first-step formulation to achieve a more accurate approximation of total travel time given a set of fragments (Section \ref{sec: enhancement 1}).
    \item \textit{Enhancement of the search process:} Integration of the dichotomic method to partition the objective space, followed by application of a checker to evaluate each partition (Section \ref{sec: enhancement 3}).
\end{itemize}


\subsubsection{Enhance the first-step formulation} \label{sec: enhancement 1}
In this section, we propose an enhanced version for the first step formulation (as presented by equations \eqref{1.1} to \eqref{1.6}, abbreviated as the original first-step model hereafter). Our objective is to significantly improve the quality of feasible solutions generated by the \FBC\ algorithm, and thus in turn greatly improve the computational efficiency of our \SOS\ algorithm.

The main limitation of the original first-step model is that it only accounts for the travel time over fragments. This may result in a large discrepancy between the estimated and actual total travel time of a feasible solution (if one exists) composed of the selected fragments, potentially yielding a low-quality solution with an excessive total travel time. 

To address this issue, we enhance the original first-step model by providing a closer approximation to the true total travel time implied by the selected fragments. This is achieved by:

\begin{enumerate}
    \item \textit{Adding lower bounds on travel times}: we incorporate lower bounds for the travel times required to connect origin depots to fragments, between fragments, and from fragments to destination depots.
    
    \item \textit{Refining lower bounds via variable fixing}: we further strengthen the lower bounds by eliminating infeasible paths in advance. That is, we fix some variables $y_{i,j}$ to zero when the arc from $i$ to $j$ violates time window constraints. Specifically:
    \[
    \textit{(Variable fixing)}:\quad \text{For } i \in D \text{ and } j \in P, \text{ if } e_i + t_{i,j} + s_i \geq l_j, \text{ then } y_{i,j} \text{ is fixed to zero.}
    \]
\end{enumerate}

These enhancements help reduce the gap between the estimated and actual travel times, leading to higher-quality feasible solutions and improved performance of the \SOS\ Algorithm.

\noindent \textbf{Enhanced first-step model to calculate enhanced lower bound of total travel time:}

\begin{align}
\label{enhanced obj}
     \min \sum\limits_{f \in \mathcal{F}} T_f x_f + \sum\limits_{i \in D \cup O}\sum\limits_{j \in P \cup F}t_{i,j}y_{i,j}
\end{align}
s.t.
\begin{align}
\label{new cons1}
    \sum \limits_{j \in P\cup F}y_{i,j} & = \sum \limits_{f \in \mathcal{F}_i}x_{f} 
    && \forall i \in D 
    \\
%
\label{new cons2}
    \sum \limits_{i \in D \cup O}y_{i,j} & = \sum \limits_{f \in \mathcal{F}_j}x_{f} 
    && \forall j \in P 
    \\
%
\label{from origin}
    \sum \limits_{j \in  P\cup F} y_{i,j} & = 1
    &&\forall i \in O 
    \\
%
\label{to destination}
    \sum \limits_{j \in  D \cup O} y_{i,j} & = 1
    && \forall i \in F 
    \\
%
    \sum \limits_{f \in \mathcal{F}} a_{i,f}x_f & = 1 
    && \forall i \in P
    \\
%
    \sum \limits_{f \in \mathcal{F}} R_f x_f & < y_2^T
    \\
%
    \sum \limits_{f \in \mathcal{F}} R_f x_f & \geq y_2^B
    \\
%
    \sum \limits_{f \in \xi} x_f & \leq |\xi| - 1 
    && \forall \xi \in \Xi
    \\
%
    y_{i,j} & \in \{0,1\}
    && \forall i \in D \cup O, j \in P \cup F
    \\
%
    x_f & \in \{0,1\}
    && \forall f \in \mathcal{F}
\end{align}
Here, $y_{i,j}$ is a binary variable that equals one if an arc is from $i$ to $j$. Distinguishing from the complete set of feasible fragments $\mathcal{F}$, we denote $\mathcal{F}_i$ the set of fragments that end at node $i$ and $\mathcal{F}_j$ the set of fragments that start at node $j$. $t_{i,j}$ is the travel time from node $i$ to $j$. Constraints (\ref{new cons1}) indicate that  from each ending node $i$ of selected fragments, there must be an arc connect it to a pickup node/ destination depot. Similarly, constraints (\ref{new cons2}) indicate that there must be an arc from a drop-off node/ origin depot to each starting node $j$ of selected fragments.

\begin{exmp}
We take the example with the top-left rectangle shown in Figure \ref{CPU on rectangles} to illustrate the computational efficiency of the enhanced formulation and the original one on the average objective values. The results are summarized in Table \ref{enhance checker results 2}. The first column reports the average objective values of enhanced/original first-step formulation (denoted as $\overline{Obj}$) while the second column reports the real solution costs (denoted as $\overline{Real.sol}$) corresponding to the feasible solutions constructed in step 2. The third column computes the gap between $\overline{Obj}$ and $\overline{Real.sol}$, and the fourth column records the number of solutions searched during the time limit (the time limit is set to 30 seconds). Surprisingly, the enhanced \SOS algorithm finishes checking this rectangle in 15.4 seconds. The average gap between the solution costs obtained from the first stage and its real solution cost is only 3.7\%. Also, by eliminating subtours, the number of feasible combinations of fragments is significantly reduced, which is another reason for this huge speed-up.   
\end{exmp}

\begin{table}[!ht]
\centering
\begin{threeparttable}
\caption{Results with/without enhancements on first-step model \\}
\label{enhance checker results 2}
\begin{tabular}{c c c c c c}
\toprule
& $\overline{Obj}$ & $\overline{Real.sol}$ & $\overline{Gap}$ & Numb sols & CPU(s)\\
\midrule
enhanced version &298.35  &310.09  & 3.7\% &170  & 15.4\\
original version& 186.096 & 310.462 & 40\% & 262 &30\\
\bottomrule
\end{tabular}
\end{threeparttable}
\end{table}

\subsubsection{Enhance \SOS algorithm with dichotomic method} \label{sec: enhancement 3}

We also consider enhancing the \SOS algorithm by applying the dichotomic algorithm \citep{aneja1979bicriteria} at the beginning to generate non-dominated areas. The benefit of such enhancement is that we can leverage the strengths of both methods and avoid their weaknesses:
\begin{enumerate}
    \item The dichotomic algorithm is efficient in finding supported non-dominated points. 
    \item The \SOS algorithm is efficient at evaluating rectangles with constrained values for the second objective, particularly when the second objective has small values or a narrow range;
\end{enumerate}

To integrate the dichotomic method into the \SOS\ framework, it suffices to make the following modification: Once the efficient solutions $z^T = z^1, z^2, \ldots, z^n = z^B$ are obtained using the dichotomic method, the initialization step (lines 1--3) of Algorithm~\ref{alg:sos_algorithm} should be updated as follows:
\begin{itemize}
    \item Set $Rec \leftarrow [R(z^1, z^2), R(z^2, z^3), \ldots, R(z^{n-1}, z^n)]$;
    \item Set $\mathcal{X}_E \leftarrow \{z^1, z^2, \ldots, z^n\}$.
\end{itemize}
This modification enables the \SOS\ algorithm to utilize the output of the dichotomic method as its initial set of solutions and regions for further exploration.

\section{Computational results} \label{sec::computational results}
We investigate the efficiency of the proposed algorithm and enhancements through experimentation. All algorithms are coded in Julia 1.11.2 and are run on a standard PC with an Intel(R) Core(TM) i7-8700 CPU at 3.20 GHz and with 32 Gb of RAM using a single thread only. The remainder of this section is organized as follows. Section \ref{benchmark algorithms} presents the benchmark algorithms considered in the computational experiments, which include two criterion space search algorithms and one decision space search algorithm. Section \ref{experimental settings} describes experimental settings, considered instances, and the quality indicators for evaluating the quality of approximated Pareto frontiers obtained from the different solution methods, whenever a time limit is reached. Section \ref{sec:impact of enhancement} evaluates the enhancement methods. Section \ref{sec:overall results on DARP} examines the performance of the proposed method on classical DARP instances where benchmark results exist in the literature. After validating the performance of our proposed \SOS algorithm on classical DARP instances, we solve the BO-EADARP with the proposed method in Section \ref{sec:overall results on E-ADARP}. As there are no existing benchmark results for the BO-EADARP, we apply the $\epsilon$-constraint method and build the reference set with all the considered algorithms. By analyzing the obtained Pareto-optimal solutions, we yield managerial insights that could be helpful in the decision-making process. 

\subsection{Benchmark algorithms} \label{benchmark algorithms}
We benchmark our \SOS algorithm with methods from the literature, i.e., the $\epsilon$-constraint method \citep{laumanns2006efficient}, the balanced box method (\cite{boland2015criterion}), and an adapted bi-objective branch-and-price (BOBP) algorithm (\cite{parragh2019branch}). For the sake of organization, we briefly present their main ideas as follows and move all details to \ref{appendix: implementation for benchmark}. At the end of this part, we compare our algorithm with the benchmark algorithms from a methodological perspective.

\subsubsection{Epsilon-constraint method} \label{epsilon-constraint method}
The mechanism of the $\epsilon$-constraint method is illustrated in Figure \ref{epsilon figure}, where the solid black points are obtained non-dominated points, the solid red point is the next non-dominated point, and hollow black points are non-dominated points that are not yet obtained. 
The $\epsilon$-constraint method starts in the rectangle area which contains all non-dominated points (e.g., $R(z^T, z^B)$ in Figure \ref{epsilon figure}). It always optimizes one objective (e.g., $z_1(x)$) while the other is bounded by an $\epsilon$ value  (i.e., $z_2(x) < \epsilon$). In each iteration, the $\epsilon$ value is updated with the $z_2(x')$, where $x'$ is the newly-found non-dominated solution. 
In numerical experiments, we test two versions of the $\epsilon$-constraint method: one uses a commercial solver (e.g., Gurobi), and the other applies an adapted B\&P algorithm from \cite{su2024branch} to solve each single-objective problem. The implementation of the second version can be found in \ref{appendix: implementation for epsilon with BP}.
\begin{figure}[H]
\centering
\includegraphics[width=5cm]{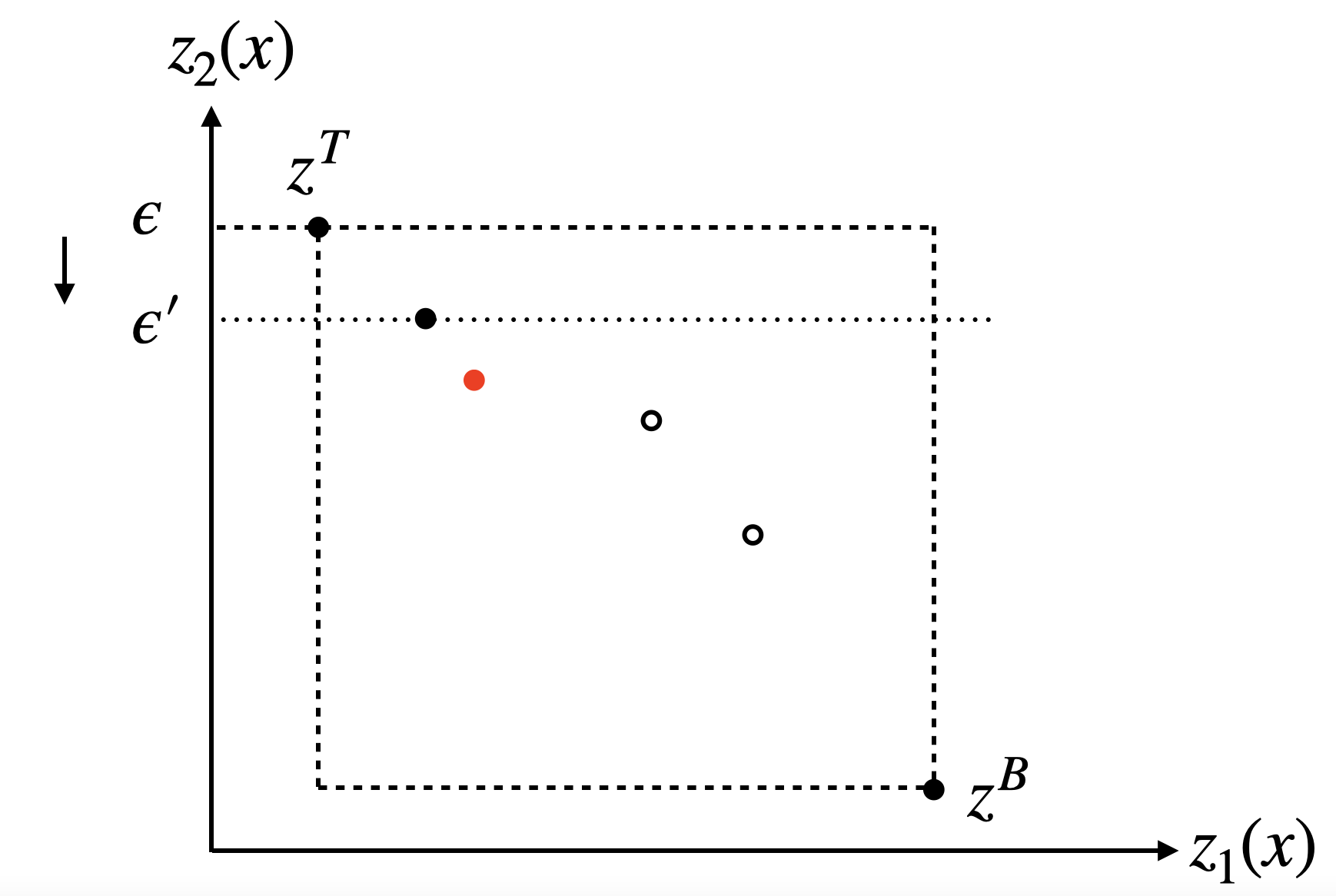}
\caption{\centering Applying the $\epsilon$-constraint method to find non-dominated points}
\label{epsilon figure}
\end{figure}

\subsubsection{Balanced Box Method} \label{sec::balanced box method}
The balanced box method starts by solving two objectives in lexicographical order (i,e., $lex\min_{x\in \mathcal{X}}\{z_1(x), z_2(x)\}$ and $lex\min_{x\in \mathcal{X}}\{z_2(x), z_1(x)\}$) to define an initial rectangle. The initial rectangle is then split horizontally to create two smaller rectangles. In each smaller rectangle, a specified lexicographical minimization problem is solved to obtain the next non-dominated point, which is then used to further partition the objective space. By iteratively splitting and searching in each well-partitioned rectangle area, all the non-dominated points can be obtained.
Similarly, we test two different versions of the balanced-box method: one calls a commercial solver (Gurobi) to solve the lexicographical optimization problem, and the other leverages the B\&P algorithm of \cite{su2024branch}. For further explanation and implementation details, we refer to \ref{appendix: implementation for balanced-box}.

\subsubsection{The BOBP algorithm}
We combine the B\&P algorithm developed in \cite{su2024branch} and the bi-objective branch-and-bound algorithm developed in \cite{parragh2019branch} to handle the BO-EADARP. Most ingredients of the proposed algorithm are taken from \cite{parragh2019branch}, while we apply the CG algorithm of \cite{su2024branch} to compute the lower bound set at each node of the B\&P tree. All technical details can be found in \ref{implementation of the BOBP}.

\subsection{Organization of experiments and settings} \label{experimental settings}
Our experiments are organized as follows. In Section \ref{sec:impact of enhancement}, we analyze the impact of each enhancement method on the overall performance. The abbreviations of different versions of the \SOS algorithm are summarized as following:
\begin{enumerate}
    \item The original \SOS algorithm without any enhancement method is abbreviated as $ori$
    \item The \SOS algorithm with enhanced model is abbreviated as $c_1$
    \item The \SOS algorithm enhanced by the dichotomic method is abbreviated as $dic$
    \item The \SOS algorithm with the above two enhancements is abbreviated as $hb$
\end{enumerate}

In Section \ref{sec:overall results on DARP} and \ref{sec:overall results on E-ADARP}, we conduct experiments with the best-performing \SOS versions and we keep their abbreviations (i.e., version $c_1$ and $hb$).
Eight different methods, including the benchmark methods presented in Section \ref{benchmark algorithms}, are considered in our experiments. To do a relatively fair comparison, these methods are classified into two groups according to their single-objective solver engines:
\begin{enumerate}
    \item  Algorithms that use a general-purpose MIP solver as search engine. They are:
    \begin{enumerate}
        \item The \SOS algorithm relying on a MIP solver to solve the second stage model (i.e., version $c_1$)
        \item The $\epsilon$-constraint method using a MIP solver to solve each single-objective problem (abbreviated as $\epsilon_1$)
        \item The balanced-box method using a MIP solver (abbreviated as $b_1$)
    \end{enumerate}
    \item  Algorithms that use a problem-tailored B\&P/CG algorithm \citep{su2024branch} as search engine. They are:
    \begin{enumerate}
        \item The \SOS algorithm with B\&P to solve the second stage model (abbreviated as $c_2$)
        \item The $\epsilon$-constraint method using B\&P to solve each single-objective problem (abbreviated as $\epsilon_2$)
        \item The balanced-box method using B\&P (abbreviated as $b_2$)
        \item The BOBP method using the dichotomic method in combination with the CG algorithm of \cite{su2024branch} at each B\&P node to obtain the lower bound set (abbreviated as $bp$)
        \item The \SOS algorithm hybridized with the dichotomic method (i.e., version $hb$)
    \end{enumerate}
\end{enumerate}

These algorithms are compared on the type-a instances of \cite{cordeau2006branch} under a bi-objective context. For these instances, we use the results reported in \cite{parragh2009heuristic} as reference results (abbreviated as $r$) to evaluate the performance of our proposed methods. The time limit per instance for each considered algorithm is set to 5 hours. In the result tables, we use the following abbreviations:
\begin{enumerate}
    \item CPU: the CPU time in seconds for solving the corresponding problem with the considered algorithm.
    \item $\mathcal{X}_E^{a}$: the set of efficient solutions obtained by the considered algorithm $a \in \{ori, dic, c_1, c_2, \epsilon_1, \epsilon_2, bp, hb, b_1, b_2, r \}$
    \item $|\mathcal{X}_E^{a}|$: the number of efficient solutions for an instance with an algorithm $a \in \{ori,dic,c_1, c_2, \epsilon_1, \epsilon_2, bp, hb, b_1, b_2, r \}$.
\end{enumerate}
If an instance is solved to optimality within the time limit, then we indicate it by an asterisk in column ``$|\mathcal{X}_E^{a}|$". For some instances, none of the considered algorithms can give the complete Pareto frontier. In other words, all the algorithms give an approximation of the Pareto frontier on these instances. To compare the performance of the algorithms in this case, we report the values of two quality indicators proposed in the literature \citep{knowles2006tutorial}:\
\begin{enumerate}
    \item The hypervolume indicator $I_H$ measures the hypervolume of the objective space that is weakly dominated by an approximation set. Higher $I_H$ values are preferable.
    \item The multiplicative $\epsilon-$indicator $I_{\epsilon}$ is defined as the minimum factor $\epsilon$ such that if every point in reference set $R$ was multiplied by $\epsilon$, the resulting approximation set would be weakly dominated by approximation set $A$. For the DARP instances, we take the results of \cite{parragh2009heuristic} as $R$. For the E-ADARP instances, we take the complete Pareto front obtained from a considered algorithm as $R$.
    Lower $I_{\epsilon}$ is preferable.
    \begin{align}
        I_{\epsilon}(A,R) = \inf \{\epsilon \in \mathbb{R}: \forall z^2 \in R ~ \exists z^1 \in A \text{ so that } z^1 \leq_{\epsilon} z^2\}
    \end{align}
\end{enumerate}

To have a clear comparison, we use normalized objective values to compute $I_H$ and $I_{\epsilon}$, we consider the upper bounds on both objectives as the reference point (denoted as $(C,L)$): 
\begin{align}
    C := \sum \limits_{i} [\max_j(t_{i,j})], \quad \forall i \in V \nonumber \\ 
    L := \sum \limits_i m_i, \quad \forall i \in P \nonumber
\end{align}
where $m_i$ is the maximum user ride time for request $i$. In the following section, we mark the best obtained $I_H$ and $I_{\epsilon}$ values in bold. If an instance is solved optimally within the time limit, we mark the results of the fastest algorithm in bold. In case a considered algorithm does not obtain an efficient solution within the time limit, we indicate this by ``NA".

\subsection{Impact of enhancement methods} \label{sec:impact of enhancement}
Before conducting extensive experiments, we assess the efficiency of the enhancement methods described in Section \ref{sec::enhancement on algorithm}. We run a version of the \SOS algorithm with no enhancement at all (abbreviated as $ori$), a version with the enhanced first-step formulation only (abbreviated as $c_1$), a version using the dichotomic method to generate initial partitions only (abbreviated as $dic$), and a version with both enhancements (abbreviated as $hb$). We evaluate the proposed enhancement methods on the small-to-medium-sized instances also solved in \cite{parragh2009heuristic} with a time limit of two hours per instance and summarize the obtained results in Table \ref{results of enhancement}.
\begin{table}[!htp]
    \centering
    \begin{threeparttable}
    \caption{Results on small-to-medium-sized bi-objective DARP instances with different enhancement methods}
    \label{results of enhancement}
    \setlength{\belowcaptionskip}{0.5cm}
    \setlength\tabcolsep{1pt}
    \footnotesize
    \begin{tabular}{c |c c c c| c c c c|c c c c| c c c c | c c c}
    \hline
    &\multicolumn{4}{c|}{\SOS (no enhancem.)} &\multicolumn{4}{c|}{\SOS + enhan.~model} &\multicolumn{4}{c|}{\SOS + dichot.} & \multicolumn{4}{c|}{\SOS + both} &\multicolumn{3}{c}{\resizebox{2.5cm}{!}{\cite{parragh2009heuristic}$^a$}}\\
        \hline
        Instance & $|\mathcal{X}_e^{ori}|$ & CPU(s) 
        & $I_H$ & $I_{\epsilon}$
        & $|\mathcal{X}_e^{c_1}|$ & CPU(s) 
        & $I_H$ & $I_{\epsilon}$
        & $|\mathcal{X}_e^{dic}|$ & CPU(s) 
        & $I_H$ & $I_{\epsilon}$
        & $|\mathcal{X}_e^{hb}|$ & CPU(s)
        & $I_H$ & $I_{\epsilon}$
        & $|\mathcal{X}_e^r|$
        & $I_H$ & $I_{\epsilon}$\\
        \hline
        a2-16 &14$^*$ &31 &0.533 &1.000
        &\textbf{14$^*$}  & \textbf{8}
        &\textbf{0.553} & \textbf{1.000}
        &14$^*$ &32 &0.553 &1.000
        &14$^*$ &25
        &0.553 &1.000  
        &14$^*$ &0.553 &1.000\\
        a2-20 &17 &7200 &0.598 &1.000
        &17$^*$  & 1540   
        &0.598 & 1.000
        &15 &7200 &0.597 &1.003
        &\textbf{17$^*$} &\textbf{620}
        &\textbf{0.598} &\textbf{1.000}  
        &17$^*$ &0.598 &1.000\\
        a3-18 &14 &7200 &0.633 &1.000
        &14$^*$  &42
        &0.633 &1.000
        &12 &7200 &0.632 &1.001
        &\textbf{14$^*$} &\textbf{31}
        &\textbf{0.633} &\textbf{1.000}  
        &14$^*$ &0.633 &1.000\\
        a3-24 &16 &7200 &0.625 &1.007
        &\textbf{19$^*$}  &\textbf{3725}  
        &\textbf{0.625} &\textbf{1.000} 
        &17 &7200 &0.625 &1.001
        &19 &7200
        &0.625 &1.000 
        &19$^*$ &0.625 &1.000\\
        a3-30 &2 &7200 &0.534 &1.145
        &7   &7200  
        &0.539  &1.111
        &12 &7200 &0.541 &1.015
        &\textbf{17} &\textbf{7200} 
        &\textbf{0.544} &\textbf{1.015} 
        &31$^*$ &0.544 &1.000\\
        a4-16 &9$^*$ &84 &0.607 &1.000
        &\textbf{9$^*$}   &\textbf{3}
        &\textbf{0.607} &\textbf{1.000}
        &9$^*$ &79 &0.607 &1.000
        &9$^*$ &11
        &0.607 &1.000
        &9$^*$ &0.607 &1.000\\
        a4-24 &18 &7200 &0.630 &1.018
        &\textbf{36$^*$}  &\textbf{70}
        &\textbf{0.631} &\textbf{1.000}
        &10 &7200 &0.572 &6.494
        &36$^*$ &459  
        &0.631 &1.000 
        &36$^*$ &0.631 &1.000\\
        a4-32 &19 &7200 &0.643 &1.039
        &\textbf{72$^*$}  &\textbf{1243}
        &\textbf{0.644} &\textbf{1.000}
        &13 &7200 &0.644 &1.017
        &72$^*$ &2611
        &0.644 &1.000  
        &NA &0.644 &1.000\\
        \hline
        Avg 
        & 13.6 & 5514.4 & 0.600 & 1.026
        & 23.5 & 1728.9 & 0.604 & 1.014
        & 12.8 & 5491.4 & 0.596 & 1.691
        & 24.8 & 2269.6 & 0.604 & 1.002
        & 26.5 & 0.604 & 1.000 \\
        \hline
    \end{tabular}
    \begin{tablenotes}
        \item[a] Ref: For instance a4-32, ``NA'' indicates missing results in \cite{parragh2009heuristic}. Reference values of $I_H$ and $I_{\epsilon}$ are computed from all solutions reported in their study.
    \end{tablenotes}
   \end{threeparttable}
\end{table}

From Table \ref{results of enhancement}, we have the following observations: (1) The \SOS algorithm with the enhanced first-step formulation achieves huge improvements in both solution quality and computational efficiency compared to the version without any enhancement. This demonstrates that the tightened formulation greatly accelerates the \FBC algorithm by eliminating a large number of feasible fragment combinations that do not contribute to new efficient solutions. This enhancement is particularly crucial when exploring rectangles associated with a wide range of excess user ride times; 
(2) The \SOS algorithm enhanced with the dichotomic method ($dic$) seems to degrade overall performance: it leads to longer computation times and a lower number of efficient solutions. This suggests that solely enhancing the definition of initial rectangles with supported efficient solutions from the dichotomic method does not speed up the search within each rectangle.  Therefore, we do not consider this version in the remaining experiments.
(3) The \SOS algorithm incorporating both enhancements ($hb$) outperforms the version enhanced solely with the tighter first-step formulation ($c_1$) on certain instances. For example, we obtain more efficient solutions for instance a3-30 and lower computation times for instances a2-20 and a3-18. This might indicate a potential benefit of integrating the dichotomic method.
Nevertheless, the \SOS algorithm with the enhanced first-step formulation only continues to perform consistently well across all instances considered. Based on the above findings, we use version $c_1$ (so-called ``checker with MIP”) and version $hb$ (so-called ``hybrid checker”) in the following experiments. 
(4) Compare with other instances, the instance a3-30 seems to be challenging to solve with our methods (\SOS with MIP as well as hybrid \SOS). We observe long computation times within each rectangle, likely caused by the larger number of fragment combinations that have to be enumerated before a new efficient solution is identified. A potential perspective could be enhancing the \SOS method with efficient heuristics or investigating a bi-directional version of the \SOS method.

\subsection{Comparison to other approaches on DARP instances} \label{sec:overall results on DARP}
We apply our proposed methods and the benchmark algorithms on DARP instances. The computational results are summarized in Tables \ref{results summarize 1: checker on DARP} and \ref{results summarize 2: checker on DARP}.

\begin{table}[!htp]
    \centering
    \begin{threeparttable}
    \caption{Comparison of bi-objective DARP solution approaches on type-a instances using a MIP solver as search engine}
    \label{results summarize 1: checker on DARP}
    \setlength{\belowcaptionskip}{0.5cm}
    \setlength\tabcolsep{2pt}
    \footnotesize
    \begin{tabular}{c |c c c c| c c c c|c c c c| c c c c}
    \toprule
        &\multicolumn{4}{c}{Checker with MIP} &\multicolumn{4}{c}{$\epsilon$ with MIP} &\multicolumn{4}{c} {Balanced-box with MIP} & \multicolumn{4}{c}{\cite{parragh2009heuristic}$^a$}\\
        \hline
        Instance & $|\mathcal{X}_e^{c_1}|$ & CPU(s) 
        & $I_H$ & $I_{\epsilon}$
        & $|\mathcal{X}_e^{\epsilon_1}|$ & CPU(s) 
        & $I_H$ & $I_{\epsilon}$
        & $|\mathcal{X}_e^{b_1}|$ & CPU(s) 
        & $I_H$ & $I_{\epsilon}$
        & $|\mathcal{X}_e^r|$ & CPU(s)
        & $I_H$ & $I_{\epsilon}$\\
        \hline
        a2-16 &\textbf{14$^*$}  & \textbf{8}
        &\textbf{0.553} & \textbf{1.000}
        & 14$^*$ & 111 
        & 0.553 & 1.000
        &14$^*$  &84  
        & 0.553 & 1.000
        &14$^*$ &191
        &0.553 &1.000  \\
        a2-20 &17$^*$  & 1540   
        &0.598 & 1.000
        & 17$^*$ & 574  
        & 0.598 & 1.000
        &17$^*$ &512 
        & 0.598 & 1.000
        &\textbf{17$^*$} &\textbf{346}  
        &\textbf{0.598} &\textbf{1.000}  \\
        a2-24 &\textbf{22$^*$}  & \textbf{443} 
        &\textbf{0.565} & \textbf{1.000}
        &22$^*$ &1118  
        & 0.565 & 1.000
        &22$^*$  &1815
        & 0.565 & 1.000
        &22$^*$ &1463 
        &0.565 &1.000 \\
        a3-18 &\textbf{14$^*$}  & \textbf{42} 
        &\textbf{0.633} & \textbf{1.000}
        & 14$^*$ & 2625  
        &0.633 & 1.000
        &14$^*$  &3620  
        & 0.633 & 1.000
        &14$^*$ &378
        &0.633 &1.000  \\
        a3-24 &19$^*$  &3725  
        &0.625 &1.000
        & 14 & 18000    
        &0.623 &1.007 
        &5  &18000  
        &0.624 & 1.048
        &\textbf{19$^*$} &\textbf{1249} 
        &\textbf{0.625} &\textbf{1.000} \\
        a3-30 &7   & 18000  
        &0.539  &1.111 
        &1 &18000    
        &0.497 &1.145 
        &4  &18000    
        &0.539  &1.103 
        &\textbf{31$^*$} &\textbf{9698} 
        &\textbf{0.544} & \textbf{1.000} \\
        a3-36 &44$^*$  &17841 
        &0.609 &1.000
        &2 &18000  
        &0.555 &1.164
        &3  &18000
        & 0.603 &1.132 
        &\textbf{44$^*$} &\textbf{14174}
        &\textbf{0.609} &\textbf{1.000}  \\
        a4-16 &\textbf{9$^*$}   & \textbf{3} 
        &\textbf{0.607} & \textbf{1.000}
        & 6  & 18000 
        &0.591 &1.046 
        &5  &18000   
        &0.607  &1.050 
        &9$^*$ &253
        &0.607 &1.000\\
        a4-24 &\textbf{36$^*$}  & \textbf{70}  
        &\textbf{0.631} & \textbf{1.000}
        & 3 & 18000  
        &0.579 & 1.164
        & 4 & 18000   
        & 0.623 &1.171 
        &36$^*$ &4208  
        &0.631 &1.000 \\
        a4-32 &\textbf{72$^*$}  & \textbf{1243}
        &\textbf{0.644} & \textbf{1.000}
        & 1 & 18000
        &0.597 & 1.145
        &3  &18000   
        & 0.640 & 1.124
        & NA & NA & 0.644 & 1.000
        \\
        a4-40 &39  &18000
        &0.673 &1.015
        &1 &18000   
        &0.595 & 1.260
        &3  &18000 
        & 0.668 & 1.169
        & NA & NA
        &0.674 &1.000 \\
        a4-48 &56  &18000  
        &0.643 &1.024
        &1 &18000   
        &0.564 & 1.241
        &0  &18000
        &NA & NA
        & NA & NA
        &0.642 &1.000 \\
        a5-40 &168 &18000 
        &0.693 &1.002
        &1 &18000  
        &0.617 & 1.270
        &0  &18000 
        &NA & NA
        & NA & NA
        &0.694 &1.000 \\
        \hline
        Avg 
        &39.7 &7378 &0.616 &1.012
        &7.5 &12809 &0.613 &1.188
        &8.6 &13391 &0.606 &1.073
        &22.9 &3551  &0.617 &1.000 \\
        \hline
    \end{tabular}
    \begin{tablenotes}
        \item[a] Ref: In \cite{parragh2009heuristic}, the B\&C algorithm from \cite{ropke2007models} is called to solve the single-objective problem using the $\epsilon$-constraint method. From instance a4-32, ``NA'' indicates missing results in their paper. Reference values of $I_H$ and $I_{\epsilon}$ are computed from all solutions reported in their study.
    \end{tablenotes}
   \end{threeparttable}
\end{table}

\begin{table}[!ht]
    \centering
    \begin{threeparttable}
    \renewcommand\arraystretch{0.8}
    \caption{Comparison of bi-objective DARP solution approaches on type-a instances using B\&P/CG as search engine}
    \label{results summarize 2: checker on DARP}
    \setlength{\belowcaptionskip}{0.5cm}
    \setlength\tabcolsep{1pt}
    \footnotesize
    \begin{tabular}{c |c c c c|c c c c |c c c c|c c c c|c c c c}
    \toprule
        &\multicolumn{4}{c}{Checker with B\&P} &\multicolumn{4}{c}{Hybrid checker} &\multicolumn{4}{c}{$\epsilon$ with B\&P}&\multicolumn{4}{c}{BOBP} &\multicolumn{4}{c}{ Balanced-box with B\&P}\\
        \hline
        Instance & $|\mathcal{X}_e^{c_1}|$ & CPU(s) & $I_H$ & $I_{\epsilon}$
        & $|\mathcal{X}_e^{hb}|$ & CPU(s) & $I_H$ &$I_{\epsilon}$
        & $|\mathcal{X}_e^{\epsilon_2}|$ & CPU(s) & $I_H$ &$I_{\epsilon}$
        & $|\mathcal{X}_e^{bp}|$ & CPU(s) & $I_H$ &$I_{\epsilon}$
        & $|\mathcal{X}_e^{b_2}|$ & CPU(s) & $I_H$ &$I_{\epsilon}$\\
        \hline
        a2-16 &14$^*$  & 28
        &0.553 &1.000
        &\textbf{14$^*$} &\textbf{25}  
        &\textbf{0.553} &\textbf{1.000}
        & 14$^*$ & 297 
        &0.553 &1.000
        &14$^*$ &376 &0.553 &1.000
        &14$^*$ &396 
        &0.553 &1.000 \\
        a2-20 &17$^*$  & 4976
        &0.598 &1.000
        &\textbf{17$^*$} & \textbf{620} 
        &\textbf{0.598} &\textbf{1.000}
        & 17$^*$ & 1941
        &0.598 &1.000
        &17$^*$ & 18000 &0.598 &1.000
        &17$^*$ &2706
        &0.598 &1.000 \\
        a2-24&22$^*$  & 10324 
        &0.565 &1.000
        &\textbf{22$^*$} & \textbf{821}
        &\textbf{0.565} &\textbf{1.000}
        &22$^*$ &18000 
        &0.565 &1.000
        &18 &18000 &0.565 &1.007
        &20 &18000
        &0.564 &1.080 \\
        a3-18 &14$^*$  & 95 
        &0.633 &1.000
        &\textbf{14$^*$} & \textbf{31} 
        &\textbf{0.633} &\textbf{1.000}
        & 14$^*$ & 316   
        &0.632 &1.000 
        &14$^*$ &242.1 &0.633 &1.000
        &14$^*$ &531 
        &0.633 &1.000 \\
        a3-24 
        &19$^*$  & 11701 
        &0.625 &1.000
        &19$^*$ & 10657  
        &0.625 &1.000
        &\textbf{19$^*$} & \textbf{2990}  
        &\textbf{0.625} &\textbf{1.000}
        &18 &18000 &0.625 &1.005
        &19$^*$ &5001 
        &0.625 &1.000\\
        a3-30 &2   & 18000
        &0.534 &1.145
        &18 & 18000
        &0.544 &\textbf{1.013}
        &11 &18000   
        &0.539 &1.107
        &22 &18000 &\textbf{0.558} &1.388
        &6 & 18000
        &0.540 &1.103\\
        a3-36 &12  & 18000 
        &0.605 &1.125
        &22 &18000 
        &\textbf{0.609} &\textbf{1.016}
        &5 &18000
        &0.596 &1.168
        &30 &18000   &0.561  &\textbf{1.016} 
        &6 &18000 
        &0.603 &1.132 \\
        a4-16&\textbf{9$^*$}   & \textbf{5}
        &\textbf{0.607} &\textbf{1.000}
        &9$^*$  & 11  
        &0.607 &1.000
        & 9$^*$  & 406
        &0.607 &1.000
        &9$^*$ &127 &0.607 &1.000
        &9$^*$ &1322 
        &0.607 &1.000 \\
        a4-24 &22  & 18000 
        &0.631 &1.007
        &\textbf{36$^*$} & \textbf{459} 
        &\textbf{0.631} &\textbf{1.000}
        & 31 & 18000 
        &0.629 &1.085
        &27 &18000 &0.631 &1.009
        &19 &18000
        &0.624 &1.171\\
        a4-32  &72$^*$  & 12650
        &0.644 &1.000
        &\textbf{72$^*$} &\textbf{2611} 
        &\textbf{0.644} &\textbf{1.000}
        & 15 & 18000 
        &0.637 &1.143
        &23 &18000 &0.644 &1.014
        &14 &18000
        &0.640 &1.124 \\
        a4-40 &34  & 18000
        &\textbf{0.673} &1.015
        &18 &18000
        &\textbf{0.673} &\textbf{1.012}
        &4 &18000 
        &0.665 &1.205
        &11 & 18000  &\textbf{0.673}  &1.048
        &3 &18000 &0.667  &1.159 \\
        a4-48 &45  & 18000
        &0.643 &\textbf{1.024}
        &23 &18000 
        &\textbf{0.644} &1.035
        &4 &18000  
        &0.626 &1.208
        &21 &18000   &\textbf{0.644}  &1.035 
        &4 &18000 &0.636 &1.171 \\
        a5-40 &\textbf{128} &18000 
        &\textbf{0.693} &\textbf{1.005}
        &76 &18000 
        &\textbf{0.693} &1.011
        &6 &18000  
        &0.677 &Inf
        &32 & 18000  &\textbf{0.693}  & 1.024 
        &4 &18000 &0.686 &1.181 \\
        \hline
        Avg & 31.5 & 12214.5 & 0.616 & 1.025  
        & 27.7 & 8095.8 & 0.617 & 1.007
        & 13.1 & 11535 & 0.611 & 1.076
        & 19.6 & 14955 & 0.620 & 1.014
        & 11.5 & 13227 & 0.613 & 1.086  \\
        \hline
    \end{tabular}
   \end{threeparttable}
\end{table}

Table \ref{results summarize 1: checker on DARP} presents numerical results of three different algorithms that use a general-purpose MIP solver as a search engine. Among them, \textit{Checker with MIP} seems to be the best regarding solution quality and computational efficiency. Our findings can be summarized as follows: 
\begin{itemize}
    \item Regarding the number of efficient solutions, \textit{Checker with MIP} finds on average 39.7 efficient solutions, which is almost 5 times the number of efficient solutions found by \textit{$\epsilon$ with MIP} and \textit{Balanced box with MIP} within the time limit. 
    For the benchmark results, we refer to the ones reported in \cite{parragh2009heuristic}, where they applied an $\epsilon$-constraint method with a tailored B\&C to solve the single-objective problem in each iteration. \cite{parragh2009heuristic} only report results for those instances that could be solved to opimality and we mark some results with ``NA" to indicate that they are not available. 
    From Table \ref{results summarize 1: checker on DARP}, we observe that the benchmark algorithms (i.e., \textit{$\epsilon$-constraint with MIP} and \textit{balanced-box with MIP}) are clearly less efficient as they can only solve small instances optimally while requiring long computation times;
    \item Regarding solution quality, the results from \textit{Checker with MIP} are very close to the benchmark results on both $I_H$ and $I_{\epsilon}$ values, where optimal solutions are available and they are better than the implemented benchmark algorithms for those instances where optimal solutions are not available. 
    \item Regarding computation time, \textit{Checker with MIP} solves small-sized instances (e.g., a2-16, a2-24, a3-18, a4-16) very efficiently (within seconds) and it also performs well on medium-to-large-scale instances, especially on a4-32. The $\epsilon$-constraint method experiences a significant increase in computation time with increasing instance size and can only find one efficient solution within time limit.
\end{itemize}
Table \ref{results summarize 2: checker on DARP} presents numerical results of five different algorithms that apply the problem-tailored B\&P \citep{su2024branch} as search engine. Among the five considered algorithms, the hybrid checker method seems to be the most efficient one in terms of solution quality and computational time. To further compare the performance of the different methods, we use \textit{performance profiles} as in \cite{boland2015criterion, parragh2019branch}: The performance of an algorithm on a certain instance is assessed by the ratio of the CPU time required by that algorithm for the analyzed instance to the minimum CPU time obtained from all considered algorithms for that instance. The \textit{performance profile} is a chart where the x-axis represents performance while the y-axis represents the fraction of instances solved at a certain performance or better. If the considered algorithm does not finish within the time limit, then it does not have a performance on that instance. If there are no algorithms that finish solving an instance within the time limit, then this instance is discarded. Figure \ref{darp performance profile} shows the performance profiles of the different solution methods on the DARP instances. Among all considered algorithms, the performance profiles clearly demonstrate that our checker-based methods outperform the other methods, and that \textit{\SOS with MIP} performs the best.
\begin{figure}[!ht]
\centering
\includegraphics[width=9cm]{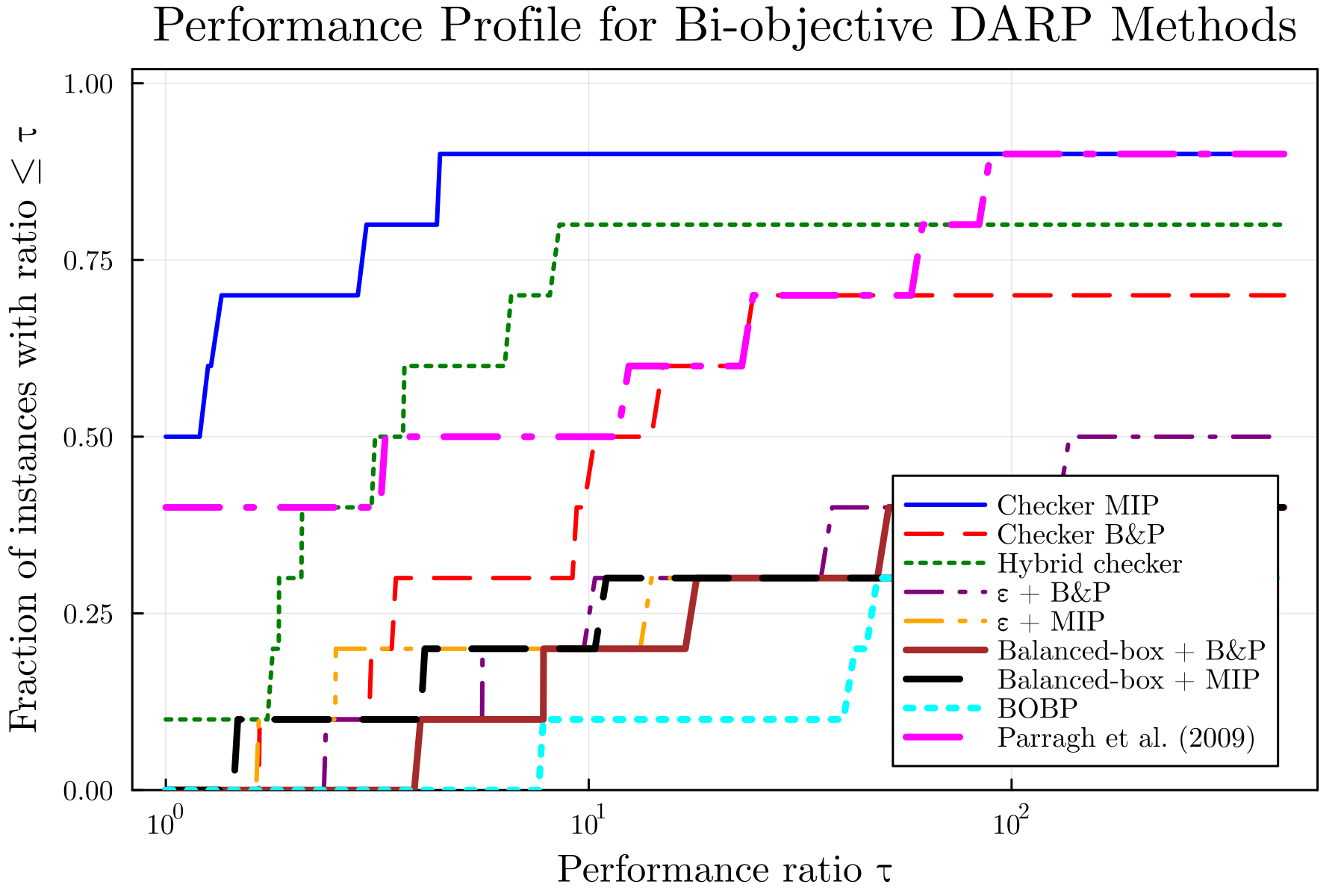}
\caption{\centering Performance profiles on DARP instances that are solved by at least one method.}
\label{darp performance profile}
\end{figure}

To summarize, the results on classical DARP instances demonstrate both the efficiency and solution quality of our proposed methods compared with state-of-the-art methods. As \textit{BOBP algorithm} and \textit{balanced-box method} are dominated by other methods, we will not consider them on the E-ADARP instances.

\subsection{Results for the E-ADARP instances} \label{sec:overall results on E-ADARP}
In this part, we apply different solution methods, including \textit{Checker with MIP}, \textit{$\epsilon$ with MIP}, \textit{Checker with B\&P}, \textit{$\epsilon$ with B\&P}, and \textit{Hybrid checker}, on the E-ADARP instances. As we do not have benchmark results from the literature, we keep all the non-dominated point candidates found by all considered methods as the reference set. We summarize results of different solution methods under different energy restriction levels ($\gamma = 0.1, 0.4, 0.7$) in Table \ref{E-ADARP results summarize} and compare them with the reference results.

\begin{table}[!ht]
\renewcommand\arraystretch{0.9}
    \centering
    \begin{threeparttable}
    \caption{Overview of bi-objective E-ADARP solution approaches with different methods}
    \label{E-ADARP results summarize}
    \setlength{\belowcaptionskip}{0.5cm}
    \setlength\tabcolsep{0.2pt}
    \footnotesize
    \begin{tabular}{c |c c c c|c c c c|c c c c|c c c c | c c c c |c c}
    \toprule
        &\multicolumn{4}{c}{Checker MIP} &\multicolumn{4}{c}{$\epsilon$ + MIP} 
        &\multicolumn{4}{c}{Checker B\&P} &\multicolumn{4}{c}{$\epsilon$ + B\&P}
        &\multicolumn{4}{c}{Hybrid checker}
        &\multicolumn{2}{c}{Ref}\\
        \hline
        $\gamma = 0.1$
        & $|\mathcal{X}_e^{c_1}|$ & CPU(s) & $I_H$ & $I_{\epsilon}$
        & $|\mathcal{X}_e^{\epsilon_1}|$ & CPU(s) & $I_H$ & $I_{\epsilon}$
        & $|\mathcal{X}_e^{c_2}|$ & CPU(s) & $I_H$ & $I_{\epsilon}$
        & $|\mathcal{X}_e^{\epsilon_2}|$ & CPU(s) & $I_H$ & $I_{\epsilon}$
        & $|\mathcal{X}_e^{hc}|$ & CPU(s) & $I_H$ & $I_{\epsilon}$
        & $|\mathcal{X}_e^{r}|$ & $I_H$\\
        \hline
        a2-16-0.1 
        &\textbf{14$^*$}  & \textbf{110}   
        &\textbf{0.581} &\textbf{1.000}
        &14$^*$ &853
        &0.581 &1.000
        &14$^*$   &157 
        &0.581 &1.000
        &14$^*$ & 6477
        &0.581 &1.000  
        &14$^*$ &143 
        &0.581 &1.000
        &14$^*$ &0.581 \\
        a2-20-0.1 
        &18 &18000
        &0.620 &1.000
        &\textbf{18$^*$} &\textbf{14695}   
        &\textbf{0.620} &\textbf{1.000}
        &17 & 18000 
        &0.620 &1.002
        &9 &18000
        &0.544 &7.160
        &13 & 18000 
        &0.620 &1.009
        &18$^*$ &0.620 \\
        a3-18-0.1 
        &14$^*$  & 2976 
        &0.654 &1.000
        &14$^*$ &3202
        &0.654 &1.000
        &\textbf{14$^*$}   & \textbf{405}
        &\textbf{0.654} &\textbf{1.000}
        &14$^*$ & 5399  
        &0.654 &1.000
        &14$^*$ & 1395
        &0.654 &1.000
        &14$^*$ &0.654 \\
        a3-24-0.1 
        &14  & 18000
        &\textbf{0.642} &1.013
        &5 &18000   
        &0.626 &1.051
        &18   & 18000
        &\textbf{0.642} &\textbf{1.006}
        &8 &18000 
        &0.585 &7.144 
        &12 & 18000
        &\textbf{0.642} &\textbf{1.006}
        &19 &0.642 \\ 
        a3-30-0.1 
        &2   & 18000  
        &0.547 &1.145
        &2 &18000 
        &0.439 &2.895
        &2    & 18000 
        &0.547 &1.145
        &3 &18000  
        &0.443 &2.895
        &18 &18000
        &\textbf{0.556} &\textbf{1.001}
        & 19 &0.556 \\
        a3-36-0.1 
        &2   & 18000 
        &0.605 &1.174
        &2 &18000
        &0.469 &17.558
        &5    & 18000 
        &0.609 &1.168
        &3 &18000 
        &0.444 &21.216
        &20 &18000
        &\textbf{0.621} &\textbf{1.001}
        &21 &0.621 \\
        a4-16-0.1 
        &9$^*$   &48
        &0.632 &1.000
        &9$^*$ &11452 
        &0.632 &1.000
        &\textbf{9$^*$}    & \textbf{19}
        &\textbf{0.632} &\textbf{1.000}
        &9$^*$ &13072
        &0.632 &1.000
        &9$^*$  & 44
        &0.632 &1.000
        &9$^*$ &0.632 \\
        a4-24-0.1 
        &\textbf{36$^*$}  &\textbf{9286}
        &\textbf{0.647} &\textbf{1.000}
        & 2 &18000 
        &0.575 &1.231
        & 36$^*$ &17340 &0.647 &1.000
        &3 &18000
        &0.453 &20.048
        &36$^*$ & 13958
        &0.647 &1.000
        &36$^*$ & 0.647\\
        a4-32-0.1 
        &33  & 18000 
        &0.655 &1.068
        & 0  &18000 
        &NA &NA
        &44  & 18000   
        &\textbf{0.656} &\textbf{1.002}
        &3 &18000  
        &0.509 &Inf
        &18 & 18000 
        &\textbf{0.656} &1.014
        &47 &0.656 \\
        a4-40-0.1 
        &23  & 18000
        &0.682 &1.113
        &0 & 18000  &NA &NA
        &30   &18000
        &\textbf{0.683} &\textbf{1.015}
        &3 &18000  
        &0.579 &Inf
        &18 & 18000 
        &\textbf{0.683} &1.018
        &31 &0.683 \\
        a5-40-0.1 
        &114 &18000
        &\textbf{0.702} &\textbf{1.006}
        &1 &18000
        &0.474 &36.550
        &108  &18000
        &\textbf{0.702} &\textbf{1.006}
        &7 & 18000
        &0.686 &1.222
        &25 &18000
        &0.701 &1.028
        &115 &0.702 \\
        \hline
        Avg 
        &25.4 &12583 &0.631  &1.047 
        &6.1 &14200 & NA &NA
        &27.0 &13083 &0.631  &1.032 
        &6.9 & 15359 & 0.558 & NA
        &17.9 &13231 &0.636  &1.007 
        &31.2 & 0.636 \\
        \hline
         $\gamma = 0.4$&\multicolumn{4}{c}{Checker MIP} &\multicolumn{4}{c}{$\epsilon$ + MIP} 
        &\multicolumn{4}{c}{Checker B\&P} &\multicolumn{4}{c}{$\epsilon$ + B\&P}
        &\multicolumn{4}{c}{Hybrid checker} 
        &\multicolumn{2}{c}{Ref}\\
        \hline
        a2-16-0.4 
        &\textbf{14$^*$}  & \textbf{125} 
        &\textbf{0.581} &\textbf{1.000}
        &14$^*$ &1848
        &0.581 &1.000
        &14$^*$   &160  
        &0.581 &1.000
        &14$^*$ &6702  &0.581   &1.000
        &14$^*$ & 149  
        &0.581 &1.000
        &14$^*$ &0.581 \\
        a2-20-0.4 
        &\textbf{16$^*$}  & \textbf{17800}
        &\textbf{0.617} &\textbf{1.000}
        &11 &18000 &0.610 &1.025
        &14   &18000 
        &0.617 &1.004
        &4 & 18000
        &0.528 &9.672
        &9  & 18000 
        &\textbf{0.617} &1.038
        &16$^*$ &0.617 \\
        a3-18-0.4 
        &14$^*$  & 6277 
        &0.654 &1.000
        &14$^*$ &3323  
        &0.654 &1.000
        &\textbf{14$^*$}   & \textbf{381}  
        &\textbf{0.654} &\textbf{1.000}
        &14$^*$ & 5137  
        &0.654 &1.000
        &14$^*$ &1238
        &0.654 &1.000
        &14$^*$ &0.654 \\
        a3-24-0.4 
        &14  & 18000 
        &\textbf{0.642} &1.013 
        &4 &18000  
        &0.624 &1.057
        &18   & 18000  
        &\textbf{0.642} &\textbf{1.006}
        &7 &18000  
        &\textbf{0.642} &\textbf{1.006}
        &10 & 18000
        &0.642 &1.017
        &19 &0.642 \\
        a4-16-0.4 
        &9$^*$   & 42
        &0.632 &1.000
        &9$^*$ &11544
        &0.632 &1.000
        &\textbf{9$^*$}    &\textbf{18}
        &\textbf{0.632} &\textbf{1.000}
        &9$^*$ & 8057 
        &0.632 &1.000
        &9$^*$  & 31 
        &0.632 &1.000 
        &9$^*$ &0.632 \\
        a4-24-0.4 
        &\textbf{36$^*$} &\textbf{13347} &\textbf{0.647}  &\textbf{1.000}
        &2 &18000
        &0.575 &1.231
        &18   & 18000 
        &0.646 &1.067
        &5 & 18000
        &0.467 &18.638
        &36$^*$ & 15226 
        &0.647  &1.000
        &36$^*$ &0.647 \\
        a4-32-0.4 
        &32  &18000
        &0.655 &1.068
        &0 &18000
        &NA &NA
        &43   &18000
        &\textbf{0.656} &\textbf{1.003}
        &4 &18000  
        &0.516 &Inf
        &22 & 18000 
        &\textbf{0.656} &1.011
        &47 &0.656 \\ 
        \hline
        Avg 
        &19.3 &10513 &0.633  &1.012 
        &7.7 &12674 &NA &NA
        &18.6 &10365 &0.633  &1.012 
        &8.1 & 13128 & 0.574 & NA
        &16.3 &10092 &0.633  &1.009
        &22.1 & 0.633 \\
        \hline
         $\gamma = 0.7$&\multicolumn{4}{c}{Checker MIP} &\multicolumn{4}{c}{$\epsilon$ + MIP} 
        &\multicolumn{4}{c}{Checker B\&P} &\multicolumn{4}{c}{$\epsilon$ + B\&P}
        &\multicolumn{4}{c}{Hybrid checker}
        &\multicolumn{2}{c}{Ref}\\
        \hline
        a2-16-0.7
        &\textbf{14$^*$}  & \textbf{360} 
        &\textbf{0.574} &\textbf{1.000}
        &14$^*$ &4342 
        &0.574 &1.000
        &14$^*$  &625  
        &0.574 &1.000
        &7 & 18000 
        &0.523 &3.685
        &14$^*$ & 1052  
        &0.574 &1.000
        &14$^*$ &0.574 \\ 
        a2-20-0.7
        &12  & 18000 
        &\textbf{0.573} &\textbf{1.000}
        &0 &18000  &NA&NA
        &10  & 18000   
        &0.545 &1.650
        & 3 & 18000    
        &0.513 &2.078
        &3  & 18000
        &0.553 &1.129
        &12 &0.573 \\
        a3-18-0.7 
        &16$^*$  & 16227  
        &0.649 &1.000
        &3 &18000
        &0.616 &1.106
        &\textbf{16$^*$}  &\textbf{1693}
        &\textbf{0.649} &\textbf{1.000}
        & 13 & 18000   
        &0.649 &Inf
        &16$^*$ & 13706  
        &0.649 &1.028
        &16$^*$ &0.649 \\ 
        a3-24-0.7 
        &12  & 18000 
        &\textbf{0.640} &\textbf{1.013}
        &2 &18000
        &0.614 &1.080
        &13  & 18000 
        &\textbf{0.640} &\textbf{1.013}
        &7 & 18000   
        &0.563 &Inf
        &11 & 18000  
        &\textbf{0.640} &1.017
        &19 &0.641 \\
        a4-16-0.7 
        &8$^*$  &408 
        &0.632 &1.000
        &4 &18000  
        &0.613 &1.057
        &\textbf{8$^*$}   & \textbf{157}
        &\textbf{0.632} &\textbf{1.000}
        &3 &18000
        &0.629 &1.111
        &8$^*$  &218
        &0.632 &1.000
        &8$^*$ &0.632 \\
        a4-24-0.7 
        &18  &18000
        &0.639 &1.071
        &0 &18000   & NA &NA
        &16  & 18000
        &\textbf{0.640} &\textbf{1.017} 
        &4 & 18000   
        &0.478 &17.038
        &11 & 18000 
        &0.639 &1.056
        &23 &0.640 \\
        a4-32-0.7 
        &8 &18000
        &0.645  &1.173
        &0 &18000   &NA &NA
        &13  & 18000
        &0.646 &1.128
        &2 &18000   &0.641  &Inf
        &12  & 18000
        &\textbf{0.649} &\textbf{1.060}
        &13 &0.650 \\
        \hline
        Avg 
        &12.4 & 12967 & 0.622 & 1.037
        &3.3 & 16048 & NA & NA
        &12.9 & 10639 & 0.618 & 1.115
        &5.6 & 18000 & 0.571 & NA
        &10.1 & 13039 & 0.619 & 1.041
        &15.0 & 0.623 \\
        \hline
    \end{tabular}
   \end{threeparttable}
\end{table}

\paragraph{Comparison among algorithms} Table \ref{E-ADARP results summarize} provides us with the following observations: (1) our checker-based methods are shown to be more efficient than the $\epsilon$-constraint method as they solve more instances optimally within shorter computational time. For instances that the considered methods cannot solve optimally within the time limit, the Pareto frontiers obtained from the checker-based methods are of better quality than those obtained from the $\epsilon$-constraint method. 
(2) It should be noted that the $\epsilon$-constraint method with MIP has nice performance on some small instances (e.g., a2-20-0.1, a3-18-0.1), where it solves them optimally within a reasonable computational time. However, the performance of $\epsilon$-constraint method degrades quickly when facing harder instances (i.e., instances of $\gamma = 0,4, 0.7$).
Like above, we draw \textit{performance profiles} to compare the different solution methods for three different cases ($\gamma = 0.1, 0.4, 0.7$), as shown in Figure \ref{performance profile eadarp}. Our checker-based methods consistently outperform the $\epsilon$-constraint method across all energy restriction levels ($\gamma = 0.1, 0.4, 0.7$), especially under the high-energy restriction.

\begin{figure}[!ht]
    \centering
    \subfigure[$\gamma = 0.1$]{
    \includegraphics[width=8.5cm]{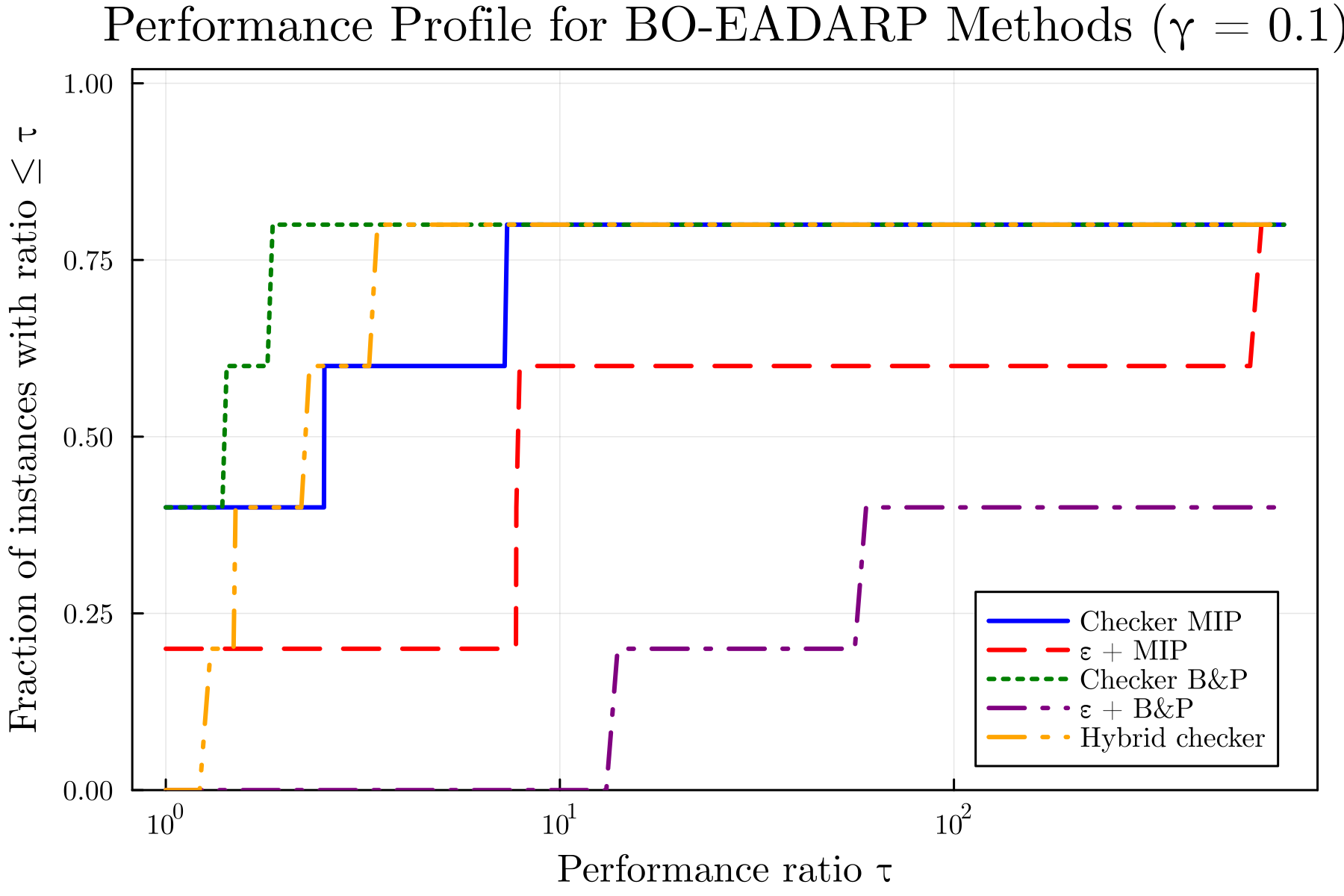}
    \label{performance profile eadarp 0.1}
    }
    \hfill
    \subfigure[$\gamma = 0.4$]{
    \includegraphics[width=8.5cm]{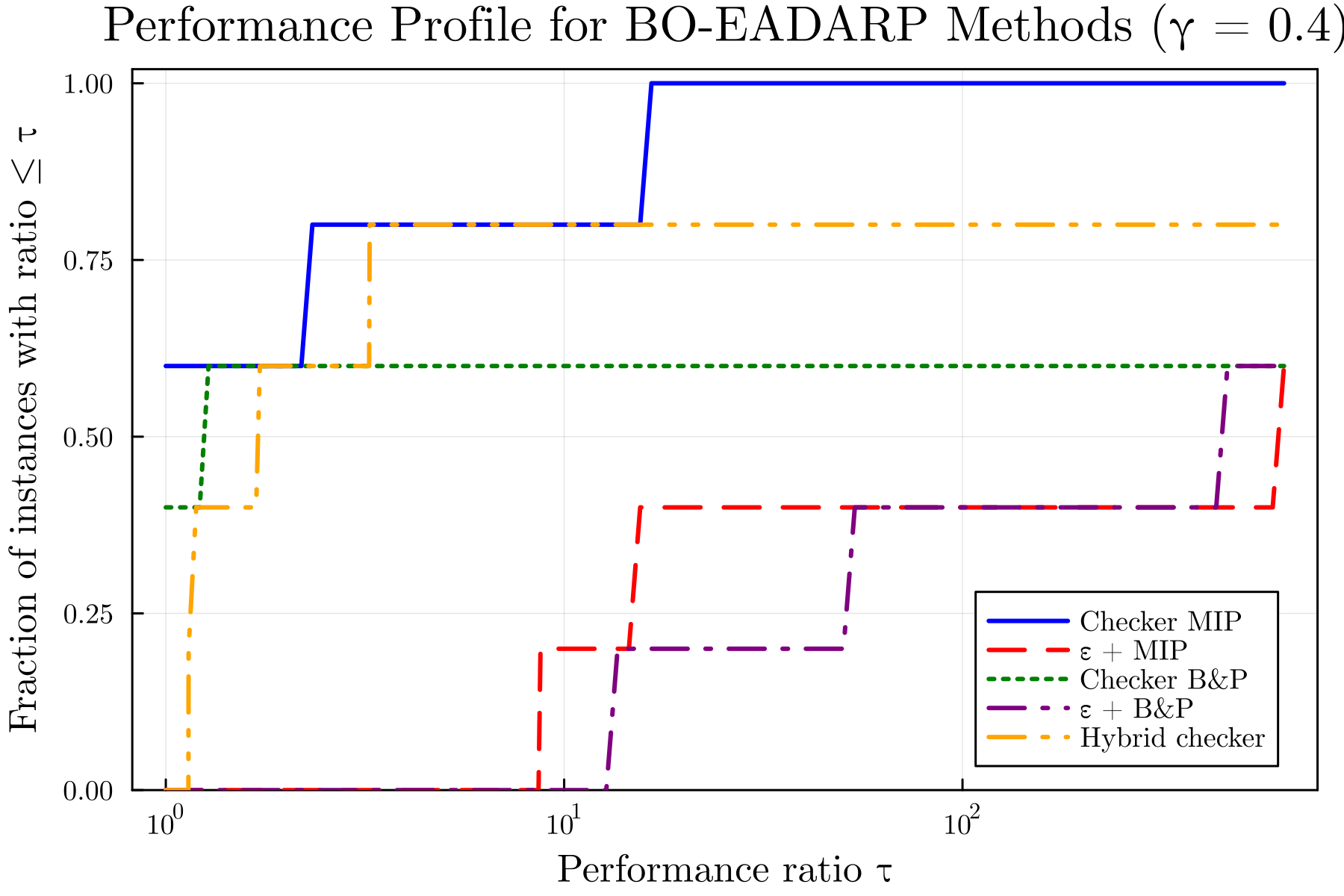}
    \label{performance profile eadarp 0.4}
    }
    \hfill
    \subfigure[$\gamma = 0.7$]{
    \includegraphics[width=9cm]{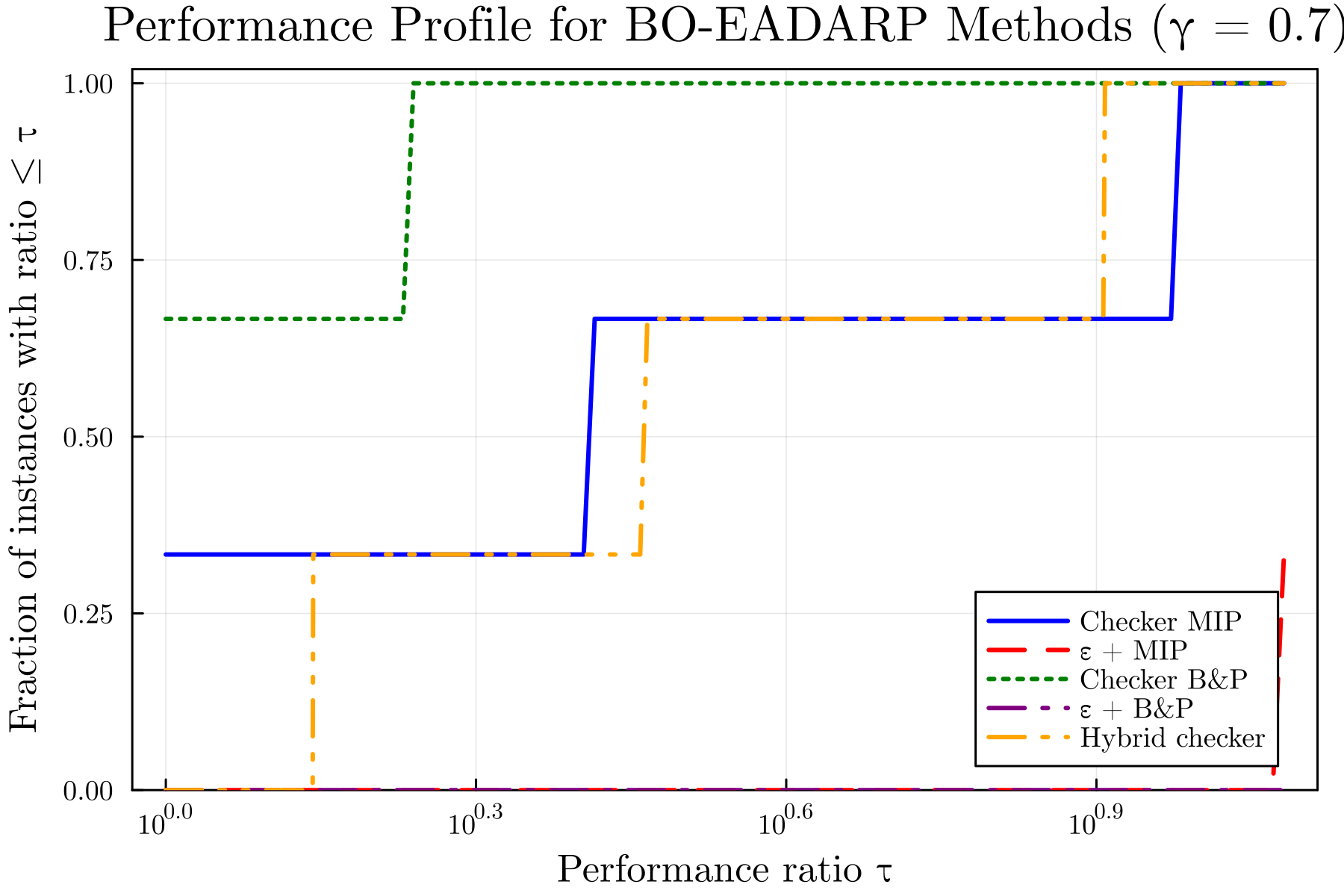}
    \label{performance profile eadarp 0.7}
    }
\caption{\centering Performance profiles on BO-EADARP instances that solved by at least one method.}
\label{performance profile eadarp}
\end{figure}




\paragraph{Comparison among different energy restriction levels}
In this part, we analyze the obtained efficient solutions under different settings of $\gamma$ for each instance, majorly from a management perspective. For the sake of explanation, we take an example with the results of instance a2-16 and a3-18 under $\gamma=0.1, 0.4,0.7$, as we generated complete sets of efficient solutions on these instance (results are shown in Figure \ref{PF under different gamma}). From the efficient solutions we obtained, we have the following observations: (1) with increasing values of $\gamma$, for most of obtained efficient solutions, there is an obvious increase in the total travel time while their total excess user ride time remains at the same level. This is consistent with the fact that vehicles need to make detours to recharging stations more frequently in order to satisfy minimum-battery-level constraints at the destination depot. As a vehicle performs recharging only when there is no passenger on board, then it will not introduce extra excess user ride time. Consequently, the efficient solutions obtained under $\gamma = 0.7$ has increased total travel times, while the total excess user ride times remain almost the same as in the case of $\gamma = 0.1,0.4$. In the case of $\gamma = 0.1, 0.4$, most of the obtained non-dominated points are the same. (2) the Pareto frontier we obtained under high end-of-route battery levels (i.e., $\gamma = 0.7$) lies inside, i.e., is dominated by, the frontiers obtained under lower energy restrictions ($\gamma \leq 0.4$). This is an expected result as some efficient solutions obtained under mild energy restriction may no longer be battery-feasible when considering more restrictive battery end levels. 

\begin{figure}[!ht]
    \centering
    \subfigure[instance (2-16)]{
    \includegraphics[width=8.5cm]{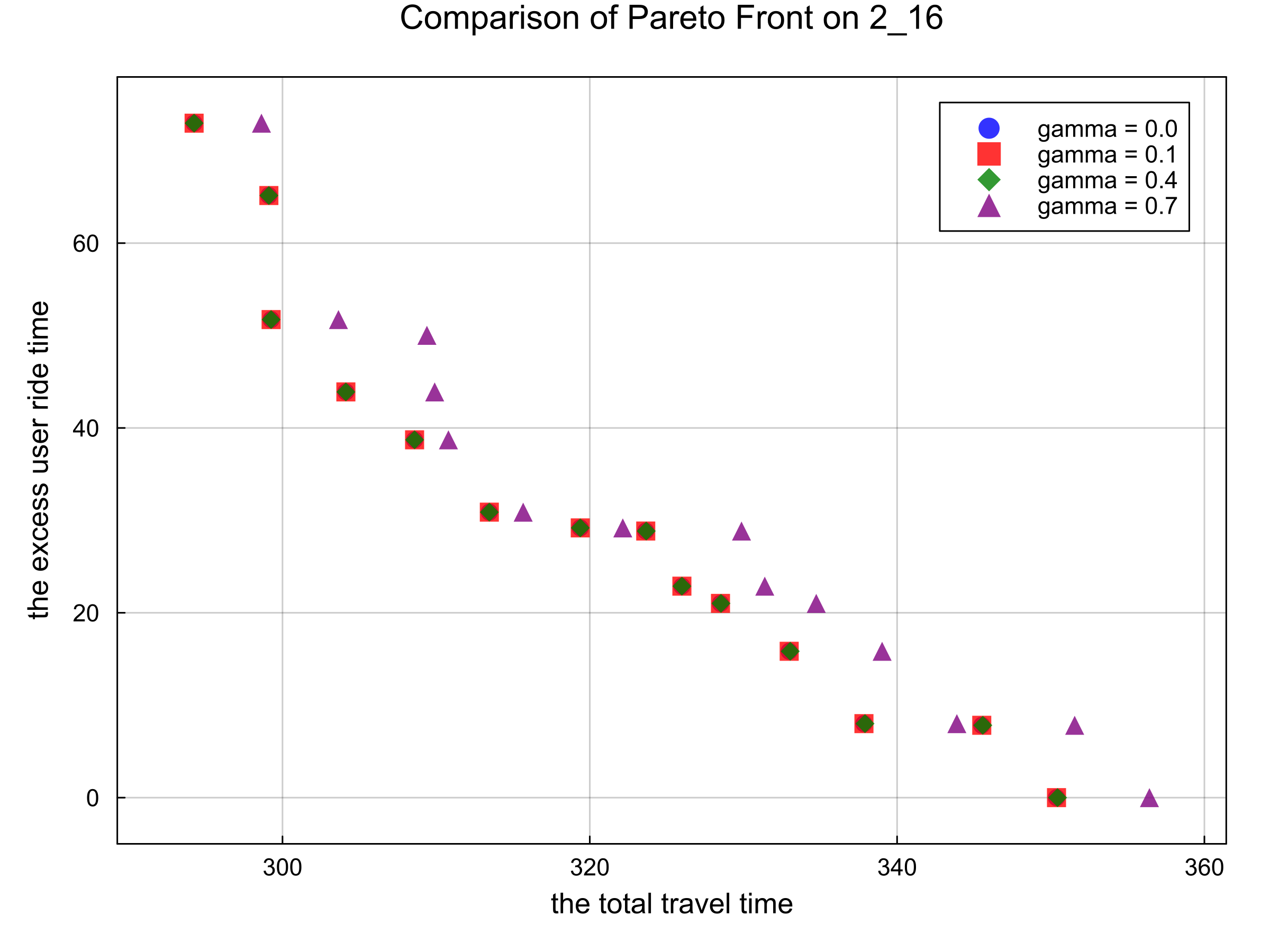}
    \label{PF diff gamma 2-16}
    }
    \hfill
    \subfigure[instance (3-18)]{
    \includegraphics[width=8.5cm]{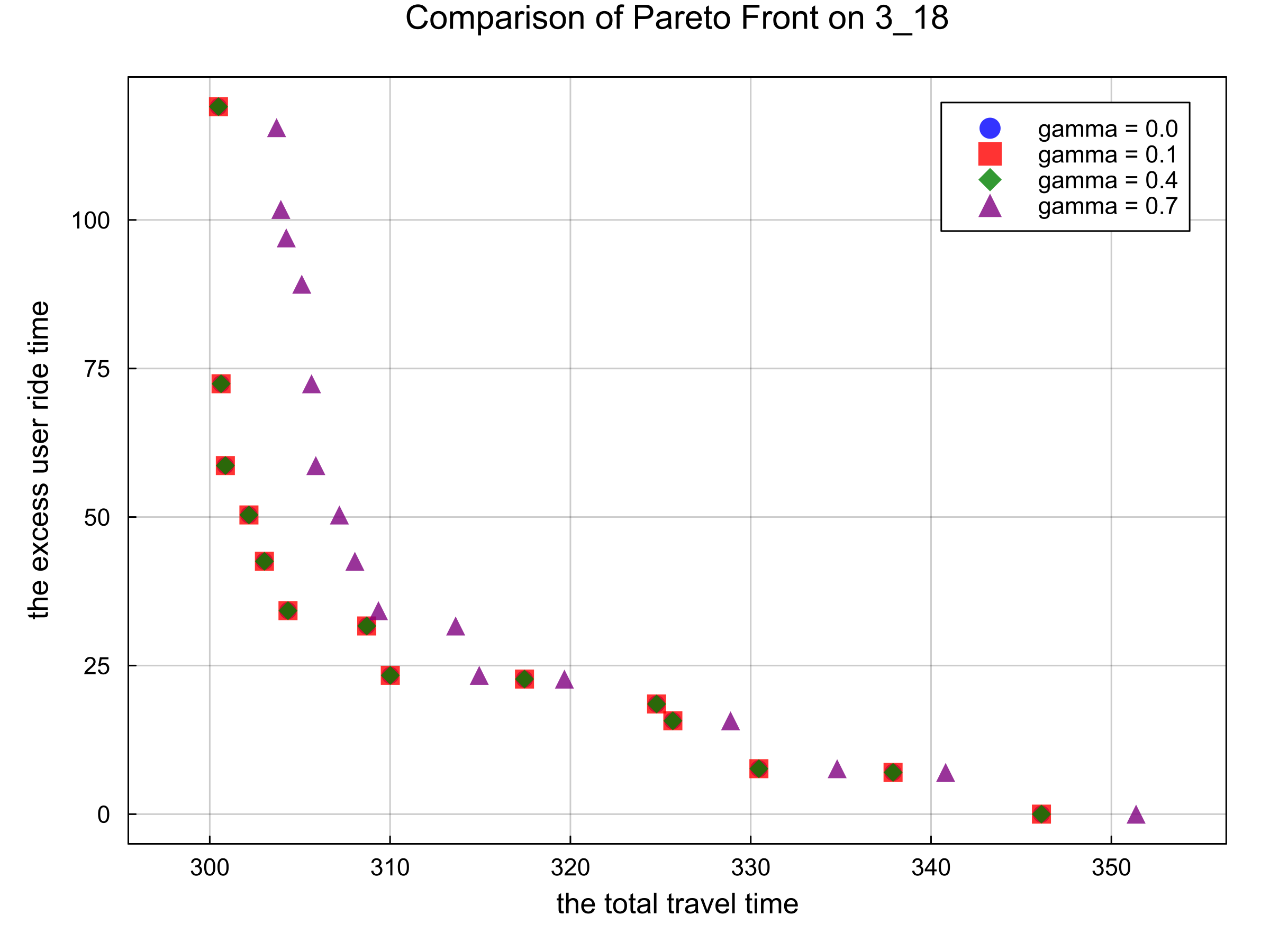}
    \label{PF diff gamma 3-18}
    }
\caption{\centering Pareto frontiers obtained under different $\gamma$ values}
\label{PF under different gamma}
\end{figure}

\subsection{Managerial insights and practical implications}
One important benefit of using the bi-objective approach is that the obtained efficient solutions can provide a clear picture of the compromises that need to be made when optimizing one objective at the expense of another. Also, efficient solutions can help decision-makers to make more informed and reliable choices. From the obtained efficient solutions on each instance, we observe a common trend, which offers us with managerial insights for different kinds of service providers. For the sake of illustration, we take the objective values of the non-dominated points for instance a2-16 under $\gamma = 0.1, 0.4, 0.7$ as an example (as shown in Table \ref{detailed results on a2-16}). The offered insights are summarized as follows:
\begin{enumerate}
    \item from the non-dominated points located in the upper left part of the Pareto front, we find that the total excess user ride time could be improved obviously with only a slight increase in the total travel time. For example, the total excess user ride time decreased from 65.15 minutes to 51.73 minutes with only 0.15 minutes increase in the total travel time (from 299.11 to 299.26 minutes) for both $\gamma=0.1$ $\gamma=0.4$. 
    This observation is of practical interests for profit-oriented service providers (e.g., Uber, Didi): it is possible to improve the service quality significantly while keeping the operational cost nearly optimal. 
    \item the non-dominated points located in the lower right part of the Pareto frontier, allow an inverse observation: in several cases, the total travel time could be improved obviously with only a slight increase in the total excess user ride time. For example, the total travel time can be reduced from 345.51 minutes to 337.84 minutes with only 0.18 minute's increase on the total excess user ride time (from 7.83 to 8.01 minutes). These efficient solutions are interesting for non-profitable associations (e.g., Red Cross) as they show that it is possible to reduce operational costs significantly while maintaining very high service quality.
\end{enumerate}

\begin{table}[!ht]
\centering
\begin{threeparttable}
\caption{Non-dominated points for instance a2-16 under $\gamma = 0.1, 0.4, 0.7$}
\label{detailed results on a2-16}
    \footnotesize
    \begin{tabular}{c c |c c | c c }
    \toprule
        $z_{0.1}^1(x)$  &$z_{0.1}^2(x)$ 
        &$z_{0.4}^1(x)$  &$z_{0.4}^2(x)$ 
        &$z_{0.7}^1(x)$  &$z_{0.7}^2(x)$ \\
         \hline
        294.25 & 72.98 & 294.25 & 72.98 & 298.63 & 72.98 \\
        299.11 & 65.15 & 299.11 & 65.15  & 303.64 & 51.73 \\
        299.26 & 51.73  & 299.26 & 51.73 & 309.40 & 50.02 \\
        304.12 & 43.90  & 304.12 & 43.90  & 309.90 & 43.90\\
        308.60 & 38.72 & 308.60 & 38.72  & 310.80 & 38.72 \\
        313.46 & 30.89 & 313.46 & 30.89 & 315.66 & 30.89 \\
        319.38 & 29.19 & 319.38 & 29.19 & 322.15 & 29.19\\
        323.64 & 28.85 & 323.67 & 28.85 & 329.87 & 28.85\\
        325.99 & 22.88 & 325.99 & 22.88 & 331.39 & 22.88\\
        328.51 & 21.02 & 328.54 & 21.02 & 334.74 & 21.02 \\
        332.98 & 15.84 & 333.05 & 15.84 & 339.02 & 15.84\\
        337.84 & 8.01  & 337.91 & 8.01  & 343.89 & 8.01\\
        345.51 & 7.83 & 345.58 & 7.83  & 351.56 & 7.83 \\
        350.38 & 0.00  & 350.44 & 0.00& 356.42 & 0.00\\
        \bottomrule
        \end{tabular}
        \end{threeparttable}
\end{table}

Few works  \citep[e.g.][]{su2024branch, stallhofer2025event} solve the weighted-sum objective and analyze the effect of setting different weights for total travel time and total excess user ride time. As solutions obtained from solving the weighted-sum objective are only supported efficient solutions (represented by points located on the convex hull of the Pareto frontier), they do not provide complete overviews of all efficient solutions (including non-supported ones). In our work, we provide decision makers with full pictures on the trade-off between operational costs and service quality, which allows them to select the appropriate transportation plans, according to their preferences.

\section{Conclusion and extensions} \label{sec::conclusion and extensions}
In this work, we work on a new problem variant of the E-ADARP, namely the BO-EADARP, where the total travel time and total excess user ride time are considered as two separate objectives. The BO-EADARP is significantly more complicated to solve than its single-objective version, as one must fully explore the bi-dimensional search space to find the complete set of Pareto optimal solutions. We propose a novel criterion-space search algorithm, the \SOSfull algorithm, to handle the BO-EADARP. The key component of the \SOS algorithm involves iteratively selecting feasible fragment combinations and checking  feasibility and dominance of the constructed solutions in a given region of the objective space (i.e., \FBCfull algorithm). To enhance the computational efficiency of the \SOS algorithm, we propose a tightened formulation and hybridize the \SOS algorithm with the dichotomic method. In numerical experiments, we first prove the efficiency of the proposed enhancement techniques. Then, we compare our proposed approaches (i.e., checker with MIP, with B\&P, and hybrid checker) with several state-of-the-art algorithms and literature results on classical DARP instances. A comparison of performance profiles highlights the efficiency 
of our proposed approaches. Then, we solve the BO-EADARP and analyze the obtained efficient solutions under different $\gamma$ settings. 
We observe a noticeable increase in the total travel times for the obtained efficient solutions with increasing $\gamma$ values, while the corresponding total excess user ride times remain stable. On each level of $\gamma$, the obtained efficient solutions offer managerial insights for different service providers: (1) for profit-oriented service providers, our results indicate that it is possible to significantly improve service quality while keeping near-optimal operational costs; (2) for not-for-profit service providers, in several cases, it is possible to improve on operational costs at rather small decreases in service quality. Our methods provide optimal Pareto frontiers for smaller scale instances and high quality approximations when hitting the time limit, allowing decision-makers to select Pareto-optimal or near-optimal transportation plans according to their priorities and preferences.

Our algorithm is a problem-tailored method for multi-objective problems that contain precedence constraints. It can therefore be applied to solve other multi-objective problems, such as the multi-objective pickup-and-delivery problem. In addition, our \SOS algorithm can be regarded as an efficient method to check the existence of efficient solutions in a given rectangle area. Therefore, it would be interesting to explore the hybridization of our algorithm with other methods. An extension could be to further enhance the performance of our algorithm with heuristic methods that can quickly find efficient solution candidates to partition the objective space. It would also
be interesting to investigate a bi-directional version.

\bibliography{main.bib}

@article{bongiovanni2019electric,
  title={The electric autonomous dial-a-ride problem},
  author={Bongiovanni, Claudia and Kaspi, Mor and Geroliminis, Nikolas},
  journal={Transportation Research Part B: Methodological},
  volume={122},
  pages={436--456},
  year={2019},
  publisher={Elsevier}
}

@article{cordeau2006branch,
  title={A branch-and-cut algorithm for the dial-a-ride problem},
  author={Cordeau, Jean-Fran{\c{c}}ois},
  journal={Operations Research},
  volume={54},
  number={3},
  pages={573--586},
  year={2006},
  publisher={INFORMS}
}

@article{ropke2007models,
  title={Models and branch-and-cut algorithms for pickup and delivery problems with time windows},
  author={Ropke, Stefan and Cordeau, Jean-Fran{\c{c}}ois and Laporte, Gilbert},
  journal={Networks: An International Journal},
  volume={49},
  number={4},
  pages={258--272},
  year={2007},
  publisher={Wiley Online Library}
}

@article{parragh2011introducing,
  title={Introducing heterogeneous users and vehicles into models and algorithms for the dial-a-ride problem},
  author={Parragh, Sophie N},
  journal={Transportation Research Part C: Emerging Technologies},
  volume={19},
  number={5},
  pages={912--930},
  year={2011},
  publisher={Elsevier}
}

@article{desaulniers2016exact,
  title={Exact algorithms for electric vehicle-routing problems with time windows},
  author={Desaulniers, Guy and Errico, Fausto and Irnich, Stefan and Schneider, Michael},
  journal={Operations Research},
  volume={64},
  number={6},
  pages={1388--1405},
  year={2016},
  publisher={Informs}
}

@article{molenbruch2017multi,
  title={Multi-directional local search for a bi-objective dial-a-ride problem in patient transportation},
  author={Molenbruch, Yves and Braekers, Kris and Caris, An and Berghe, Greet Vanden},
  journal={Computers \& Operations Research},
  volume={77},
  pages={58--71},
  year={2017},
  publisher={Elsevier}
}

@article{parragh2009heuristic,
  title={A heuristic two-phase solution approach for the multi-objective dial-a-ride problem},
  author={Parragh, Sophie N and Doerner, Karl F and Hartl, Richard F and Gandibleux, Xavier},
  journal={Networks: An International Journal},
  volume={54},
  number={4},
  pages={227--242},
  year={2009},
  publisher={Wiley Online Library}
}

@article{paquette2013combining,
  title={Combining multicriteria analysis and tabu search for dial-a-ride problems},
  author={Paquette, Julie and Cordeau, Jean-Fran{\c{c}}ois and Laporte, Gilbert and Pascoal, Marta MB},
  journal={Transportation Research Part B: Methodological},
  volume={52},
  pages={1--16},
  year={2013},
  publisher={Elsevier}
}

@article{rist2021new,
  title={A new formulation for the dial-a-ride problem},
  author={Rist, Yannik and Forbes, Michael A},
  journal={Transportation Science},
  volume={55},
  number={5},
  pages={1113--1135},
  year={2021},
  publisher={INFORMS}
}

@article{su2023deterministic,
  title={A deterministic annealing local search for the electric autonomous dial-a-ride problem},
  author={Su, Yue and Dupin, Nicolas and Puchinger, Jakob},
  journal={European Journal of Operational Research},
  year={2023},
  publisher={Elsevier}
}

@article{glize2022varepsilon,
  title={An $\varepsilon$-constraint column generation-and-enumeration algorithm for Bi-Objective Vehicle Routing Problems},
  author={Glize, Est{\`e}le and Jozefowiez, Nicolas and Ngueveu, Sandra Ulrich},
  journal={Computers \& Operations Research},
  volume={138},
  pages={105570},
  year={2022},
  publisher={Elsevier}
}

@article{parragh2019branch,
  title={Branch-and-bound for bi-objective integer programming},
  author={Parragh, Sophie N and Tricoire, Fabien},
  journal={INFORMS Journal on Computing},
  volume={31},
  number={4},
  pages={805--822},
  year={2019},
  publisher={INFORMS}
}

@article{aneja1979bicriteria,
  title={Bicriteria transportation problem},
  author={Aneja, Yash P and Nair, Kunhiraman PK},
  journal={Management Science},
  volume={25},
  number={1},
  pages={73--78},
  year={1979},
  publisher={INFORMS}
}

@article{boland2015criterion,
  title={A criterion space search algorithm for biobjective integer programming: The balanced box method},
  author={Boland, Natashia and Charkhgard, Hadi and Savelsbergh, Martin},
  journal={INFORMS Journal on Computing},
  volume={27},
  number={4},
  pages={735--754},
  year={2015},
  publisher={INFORMS}
}

@article{haimes1971bicriterion,
  title={On a bicriterion formulation of the problems of integrated system identification and system optimization},
  author={Haimes, Yacov},
  journal={IEEE transactions on systems, man, and cybernetics},
  number={3},
  pages={296--297},
  year={1971},
  publisher={Institute of Electrical and Electronics Engineers (IEEE)}
}

@article{mavrotas1998branch,
  title={A branch and bound algorithm for mixed zero-one multiple objective linear programming},
  author={Mavrotas, George and Diakoulaki, Danae},
  journal={European Journal of Operational Research},
  volume={107},
  number={3},
  pages={530--541},
  year={1998},
  publisher={Elsevier}
}

@article{masin2008diversity,
  title={Diversity maximization approach for multiobjective optimization},
  author={Masin, Michael and Bukchin, Yossi},
  journal={Operations Research},
  volume={56},
  number={2},
  pages={411--424},
  year={2008},
  publisher={INFORMS}
}

@article{vincent2013multiple,
  title={Multiple objective branch and bound for mixed 0-1 linear programming: Corrections and improvements for the biobjective case},
  author={Vincent, Thomas and Seipp, Florian and Ruzika, Stefan and Przybylski, Anthony and Gandibleux, Xavier},
  journal={Computers \& Operations Research},
  volume={40},
  number={1},
  pages={498--509},
  year={2013},
  publisher={Elsevier}
}

@article{sourd2008multiobjective,
  title={A multiobjective branch-and-bound framework: Application to the biobjective spanning tree problem},
  author={Sourd, Francis and Spanjaard, Olivier},
  journal={INFORMS Journal on Computing},
  volume={20},
  number={3},
  pages={472--484},
  year={2008},
  publisher={INFORMS}
}

@article{stidsen2014branch,
  title={A branch and bound algorithm for a class of biobjective mixed integer programs},
  author={Stidsen, Thomas and Andersen, Kim Allan and Dammann, Bernd},
  journal={Management Science},
  volume={60},
  number={4},
  pages={1009--1032},
  year={2014},
  publisher={Informs}
}

@article{ralphs2006improved,
  title={An improved algorithm for solving biobjective integer programs},
  author={Ralphs, Ted K and Saltzman, Matthew J and Wiecek, Margaret M},
  journal={Annals of Operations Research},
  volume={147},
  pages={43--70},
  year={2006},
  publisher={Springer}
}

@article{chalmet1986algorithm,
  title={An algorithm for the bi-criterion integer programming problem},
  author={Chalmet, LG and Lemonidis, L and Elzinga, DJ},
  journal={European Journal of Operational Research},
  volume={25},
  number={2},
  pages={292--300},
  year={1986},
  publisher={Elsevier}
}

@article{laumanns2006efficient,
  title={An efficient, adaptive parameter variation scheme for metaheuristics based on the epsilon-constraint method},
  author={Laumanns, Marco and Thiele, Lothar and Zitzler, Eckart},
  journal={European Journal of Operational Research},
  volume={169},
  number={3},
  pages={932--942},
  year={2006},
  publisher={Elsevier}
}

@article{forget2022enhancing,
  title={Enhancing Branch-and-Bound for Multi-Objective 0-1 Programming},
  author={Forget, Nicolas and Parragh, Sophie N},
  journal={INFORMS Journal on Computing},
  year={2024},
  volume={36},
number={1},
pages={285--304}
}

@article{mavrotas2005multi,
  title={Multi-criteria branch and bound: A vector maximization algorithm for mixed 0-1 multiple objective linear programming},
  author={Mavrotas, George and Diakoulaki, Danae},
  journal={Applied mathematics and computation},
  volume={171},
  number={1},
  pages={53--71},
  year={2005},
  publisher={Elsevier}
}

@article{ulungu1993optimisation,
  title={Optimisation combinatoire multicrit{\'e}re: D{\'e}termination de l’ensemble des solutions efficaces et m{\'e}thodes interactives},
  author={Ulungu, Ekunda L},
  journal={Unpublished Ph. D. dissertation. Facult{\'e} Polytechnique de Mons, Belgium},
  year={1993}
}

@article{przybylski2008two,
  title={Two phase algorithms for the bi-objective assignment problem},
  author={Przybylski, Anthony and Gandibleux, Xavier and Ehrgott, Matthias},
  journal={European Journal of Operational Research},
  volume={185},
  number={2},
  pages={509--533},
  year={2008},
  publisher={Elsevier}
}

@article{su2024branch,
  title={A Branch-and-Price algorithm for the electric autonomous Dial-A-Ride Problem},
  author={Su, Yue and Dupin, Nicolas and Parragh, Sophie N and Puchinger, Jakob},
  journal={Transportation Research Part B: Methodological},
  volume={186},
  pages={103011},
  year={2024},
  publisher={Elsevier}
}

@article{knowles2006tutorial,
  title={A tutorial on the performance assessment of stochastic multiobjective optimizers},
  author={Knowles, Joshua D and Thiele, Lothar and Zitzler, Eckart},
  journal={Tik report},
  volume={214},
  year={2006},
  publisher={ETH Zurich}
}

@article{stallhofer2025event,
  title={Event-based models for the electric autonomous dial-a-ride problem},
  author={Stallhofer, Verena and Parragh, Sophie N},
  journal={Transportation Research Part C: Emerging Technologies},
  volume={171},
  pages={104896},
  year={2025},
  publisher={Elsevier}
}

@article{tricoire2012multi,
title = {Multi-directional local search},
journal = {Computers \& Operations Research},
volume = {39},
number = {12},
pages = {3089-3101},
year = {2012},
issn = {0305-0548},
doi = {https://doi.org/10.1016/j.cor.2012.03.010},
url = {https://www.sciencedirect.com/science/article/pii/S0305054812000639},
author = {Fabien Tricoire},
}

@article{TRICOIRE20121582,
title = {The bi-objective stochastic covering tour problem},
journal = {Computers \& Operations Research},
volume = {39},
number = {7},
pages = {1582-1592},
year = {2012},
issn = {0305-0548},
doi = {https://doi.org/10.1016/j.cor.2011.09.009},
url = {https://www.sciencedirect.com/science/article/pii/S030505481100270X},
author = {Fabien Tricoire and Alexandra Graf and Walter J. Gutjahr}
}

\newpage
\appendix

\section{MILP formulation for step 2}
\label{appendix: stage 2 formulation}
The formulation of the \textit{Step 2} is defined on the graph $G_{sp}$ and we aim at constructing a feasible solution using the selected fragments of \textit{Step 1} while minimizing the total travel time:

\begin{equation}
     \min \sum\limits_{k \in K}\sum\limits_{i,j \in V}t'_{i,j}x_{i,j}^k
\end{equation}

subject to:

\begin{equation} \label{cons0.1}
    \sum\limits_{k \in K} x_{i,d_i'}^k = 1, \quad \forall i \in P'
\end{equation}

\begin{equation} \label{cons0.2}
    \sum\limits_{k \in K, j \in V  \setminus{d_i'}} x_{i,j}^k = 0, \quad \forall i \in P'
\end{equation}

\begin{equation} \label{cons1}
 \sum\limits_{j \in P' \cup F} x_{o^k,j}^{k} = 1, \quad \forall k \in K
\end{equation}

\begin{equation} \label{cons2}
 \sum\limits_{i \in D'\cup \{o^k\}}\sum\limits_{j \in F} x_{i,j}^{k} = 1, \quad \forall k \in K
\end{equation}

\begin{equation} \label{cons4}
 \sum\limits_{j \in V, j \neq i} x_{i,j}^{k} - \sum\limits_{j \in V, j \neq i} x_{j,i}^{k} = 0, \quad \forall k \in K, i \in N 
\end{equation}

\begin{equation} \label{cons3}
 \sum\limits_{k \in K} \sum\limits_{i \in D' \cup \{o^k\}} x_{i,j}^{k} \leqslant 1, \quad \forall j \in F 
\end{equation}

\begin{equation} \label{cons8}
\mathcal{B}^e_i \leqslant T_i^k \leqslant \mathcal{B}^l_i, \quad \forall k \in K, i \in V
\end{equation}

\begin{equation} \label{cons9}
 T_i^k + t_{i,j}' + s_i - M_{i,j}(1-x_{i,j}^k) \leqslant T_j^k,  \quad \forall k \in K, i \in V, j \in V, i \neq j| M_{i,j}>0
\end{equation}

\begin{equation} \label{cons17}
 B_{o^k}^k = B_0^k, \forall k \in K
\end{equation}

\begin{equation} \label{cons18}
 B_j^k \leqslant B_i^k - b'_{i,j} + Q(1-x_{i,j}^k), \forall k \in K, i \in V \setminus S, j \in V  \setminus \{o^k\}, i \neq j
\end{equation}

\begin{equation} \label{cons19}
 B_j^k \geqslant B_i^k - b'_{i,j} - Q(1-x_{i,j}^k), \forall k \in K, i \in V \setminus S, j \in V  \setminus \{o^k\}, i \neq j
\end{equation}

\begin{equation} \label{cons20}
 B_j^k \leqslant B_s^k + \alpha E_s^k - b_{s,j} + Q(1-x_{s,j}^k), \forall k \in K, s \in S, j \in P \cup F \cup S, s \neq j
\end{equation}

\begin{equation} \label{cons21}
 B_j^k \geqslant B_s^k + \alpha E_s^k - b_{s,j} - Q(1-x_{s,j}^k), \forall k \in K, s \in S, j \in P \cup F \cup S, s \neq j
\end{equation}

\begin{equation} \label{cons24}
    E_s^k \leqslant T_i^k - t_{s,i} - T_s^k + M_{s,i}^k\left(1-x_{s,i}^{k}\right), \quad \forall s \in S, i \in P\cup{S}\cup{F}, k\in K, i \ne s
\end{equation}

\begin{equation} \label{cons25}
    E_s^k \geqslant T_i^k - t_{s,i} - T_s^k - M_{s,i}^k\left(1-x_{s,i}^{k}\right), \quad \forall s \in S, i \in P\cup{S}\cup{F}, k\in K, i \ne s
\end{equation}

\begin{equation} \label{cons22}
 Q \geqslant B_s^k + \alpha E_s^k, \forall k \in K, s \in S
\end{equation}

\begin{equation} \label{cons23}
 B_i^k \geqslant \gamma Q, \forall k \in K, i \in F
\end{equation}

\begin{equation} \label{cons26}
    x_{i,j}^k \in \{0,1\}, \quad \forall k \in K, i \in V, j \in V
 \end{equation}

 \begin{equation} \label{cons27}
    B_i^k \geqslant 0, \forall k \in K, i \in V
 \end{equation}

 \begin{equation} \label{cons28}
    E_s^k \geqslant 0, \forall k \in K, s \in S
 \end{equation}
 
In the MILP formulation for step 2, $P'$ is the set of starting nodes of selected fragments from step 1 and $D'$ is the corresponding ending nodes set. $O$ and $F$ are origin and destination depot sets, respectively and $V$ is the set of vertices and $V = O \cup F \cup P'\cup D'$. The objective is to minimize the total travel time of the solution, where $t_{i,j}'$ denotes the travel time from node $i$ to $j$ on the graph $G_{sp}$. Constraints (\ref{cons0.1}) restrict that the arcs converted from the selected fragments must be visited and $d_i'$ is the corresponding ending node of the fragment and the starting node is $i$. Constraints (\ref{cons0.2}) forbid the visits from $i \in P'$ to other nodes except $d_i'$. Constraints (\ref{cons1}) and (\ref{cons2}) ensure that all vehicles start from their origin depots and end at a selected destination depot. Constraints (\ref{cons3}) guarantee that each destination depot is visited at most once. Constraints (\ref{cons4}) are flow conservation constraints. Constraints (\ref{cons8}) and (\ref{cons9}) are time window constraints based on adapted time windows on the graph $G_{sp}$. Constraints (\ref{cons17}) set the initial battery level for vehicle $k$ at the origin depot $o^k$. Constraints (\ref{cons18}) and (\ref{cons19}) calculate the battery levels of vehicles from $i \in V \setminus S$ to $j \in V \setminus \{o^k\}$. $b_{i,j}'$ represents the battery consumption on $G_{sp}'$. Constraints (\ref{cons20}) and (\ref{cons21}) track the battery levels of vehicles that leave from a recharging station $s \in S$ to the next node $j$. The maximum battery capacity is ensured in constraints (\ref{cons22}). Constraints (\ref{cons23}) restrict the minimum battery level that the vehicle must satisfy when returning to the destination depot. Constraints (\ref{cons24}) and (\ref{cons25}) set lower and upper bounds for recharging time at a recharging station.


\section{Implementation details for benchmark methods} \label{appendix: implementation for benchmark}
\subsection{Epsilon-constraint method integrated with B\&P to solve each single-objective problem} \label{appendix: implementation for epsilon with BP}
In the implementation of the $\epsilon$-constraint method to solve the BO-EADARP, we set the total travel time as the objective and add the total excess user ride time to the MP formulation. For each single-objective problem (with different $\epsilon$ values), we adapt the B\&P algorithm in \cite{su2024branch} by adding Constraint (\ref{chap5.3}) in the MP. Consequently, the dual variable $\kappa$ has to be accounted for when applying the labeling algorithm to solve the pricing subproblems. Briefly, when a partial path $\mathcal{P}$ is extended with a fragment $\mathcal{F}$, $R_{\mathcal{F}}\kappa$ is subtracted in the reduced cost computation where $R_{\mathcal{F}}$ is the corresponding excess user ride time on $\mathcal{F}$. The MP is modified as follows, and we denote the adapted MP as MP$^{\epsilon}$:
\begin{equation} \label{chap5.1}
\min \sum\limits_{\omega \in \Omega'}T_{\omega} y_{\omega} + \sum\limits_{i \in P}P_i a_i
\end{equation}
subject to:
\begin{equation}\label{chap5.2}
 \sum\limits_{\omega \in \Omega}\theta_{i \omega} y_{\omega} \geqslant 1 - a_i, \forall i \in P
\end{equation}

\begin{equation}\label{chap5.3}
 \sum\limits_{\omega \in \Omega}R_{\omega} y_{\omega} \leqslant \mathcal{E}_1 - \delta
\end{equation}

\begin{equation}\label{chap5.4}
 \sum\limits_{\omega \in \Omega}\phi_{f \omega} y_{\omega} \leqslant 1, \forall f \in S \cup F
\end{equation}

\begin{equation}\label{chap5.9}
     \sum\limits_{\omega \in \Omega}\epsilon_{o \omega} y_{\omega} \leqslant 1, \quad \forall o \in O
\end{equation}

\begin{equation}\label{chap5.5}
\sum\limits_{\omega \in \Omega} y_{\omega} \leqslant |K|
\end{equation}

\begin{equation}\label{chap5.6}
     y_{\omega} \in \{0,1\}, \forall \omega \in \Omega
\end{equation}

\begin{equation}\label{chap5.7}
     a_{i} \in \{0,1\}, \forall i \in P
\end{equation}

In the objective function, $T_{\omega}$ is the total travel time for route $\omega$. The Constraint (\ref{chap5.3}) is the constraint for excess user ride time. $R_{\omega}$ denotes the excess user ride time of route $\omega$. The objective function of the pricing subproblem is:
\begin{equation}\label{chap5.8}
c_{\omega} - \sum\limits_{i \in P}\theta_{i,\omega}\pi_i -\sum\limits_{f \in S \cup F}\phi_{f \omega}\tau_f - \sum\limits_{o \in O}\epsilon_{o \omega}\zeta_o - \sigma - R_{\omega} \kappa
\end{equation}
where $\pi_i \ (i \in P)$, $\tau_f \ (f\in S \cup F)$, and $\zeta_o \ (o \in O)$ are the dual variable of Constraints (\ref{chap5.2}), (\ref{chap5.4}), (\ref{chap5.9}), respectively. Another dual variable associated with Constraint (\ref{chap5.5}) is $\sigma$. The newly introduced dual variable by Constraint (\ref{chap5.3}) is denoted as $\kappa$. 

It should be mentioned that the $\epsilon$-constraint method can also be implemented in the other direction. That is, considering the total excess user ride time as the objective and bounding the total travel time with an $\epsilon$ value. However, our preliminary experiments showed that the $\epsilon$-constraint method in the other direction yields much poorer performance compared to considering the total travel time as objective. The reason is that solutions with different total travel times may have the same objective value (i.e., the same value of total excess user ride time). As a result, we must conduct several computations before reaching the one that has the minimum total travel time among the solutions that have the same value of total excess user ride time (i.e., the real non-dominated point), which introduces extensive computations. 
In contrast, using total travel time as the objective allows for more efficient and effective identification of non-dominated points, leading to better performance. 
Hence, we will only present the $\epsilon$-constraint method that considers the total travel time as the objective in this work.

\subsection{Implementation details for the balanced-box method} \label{appendix: implementation for balanced-box}
The balanced box method computes non-dominated points recursively in the initial rectangle defined by two lex-min solutions. For the sake of clarity, we use the notations $z^T$ and $z^B$ to denote the obtained non-dominated points through the lexicographical minimization processes. These two points define a rectangle, denoted as $R(z^T,z^B)$. 
\begin{enumerate}
\item First, the rectangle $R(z^T,z^B)$ is split in the horizontal direction along $z_2(x)$ axis into two equal smaller rectangles $R^T$ and $R^B$. These two smaller rectangles are defined by $z^T$ and $z^t$, $z^b$ and $z^B$, respectively, with $z^t = (z_1^B,(z_2^T+z_2^B)/2)$ and $z^b = (z_1^T,(z_2^T+z_2^B)/2)$. As shown in Figure \ref{balanced_box1}, we define two rectangles $R^T$ and $R^B$. The red solid point is the next non-dominated point that we will obtain in the next step, and the hollow black points are those that we have not obtained yet.
\item Then, rectangle $R^B$ is searched to find the next non-dominated point by solving lexicographically $z_1(x)$ and $z_2(x)$, i.e., 
$lex\min_{x\in \mathcal{X}}\{z_1(x),z_2(x):z(x) \in R^B\}$. The obtained solution $\Bar{x}^1$ is an efficient solution, and its image $\Bar{z}^1$ is a non-dominated point, as proved in \cite{boland2015criterion}. This point is used to further partition the rectangle $R^T$, as some parts in $R^T$ are dominated by $\Bar{z}^1$ (shown in Figure \ref{balanced_box2}). Then, we focus on finding the next non-dominated point (marked in red) in rectangle $R^T$.
\item Next, the rectangle $R^T$ is explored by solving lexicographically $z_2(x)$ and $z_1(x)$, i.e., 
$lex\min_{x\in \mathcal{X}}\{z_2(x),z_1(x):z(x) \in R^T\}$. Supposing the obtained solution is $\Bar{x}^2$ and its image is $\Bar{z}^2$, it defines a new rectangle $R(z^T,\Bar{z}^2)$ (as shown in Figure \ref{balanced_box3}). 
\end{enumerate}
The above process is repeated until there are no unexplored rectangles.

\begin{figure}[!ht]
    \centering
    \subfigure[Step1.]{
    \includegraphics[width=5cm]{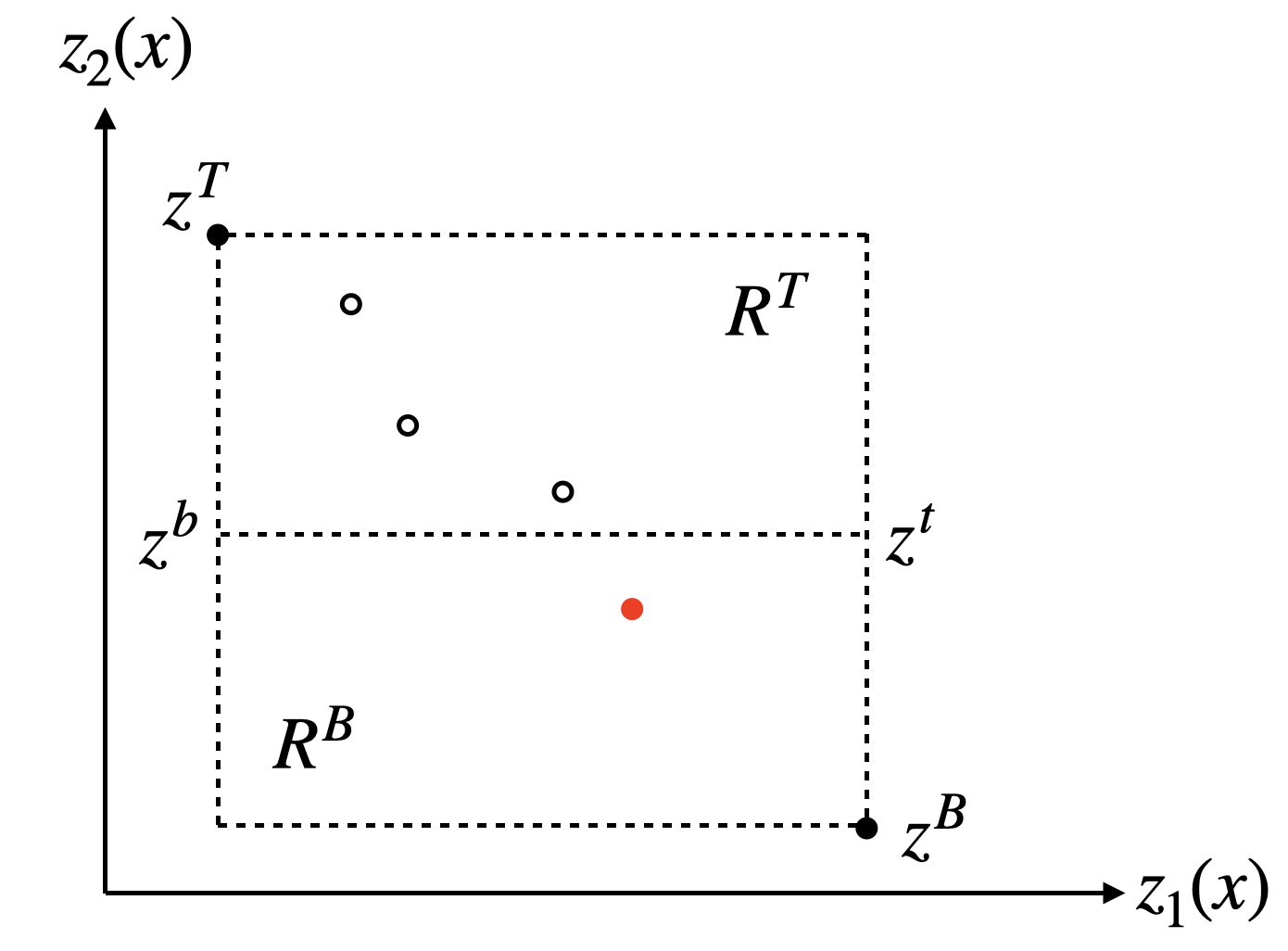}
    \label{balanced_box1}
    }
    \hfill
    \subfigure[Step2.]{
    \includegraphics[width=5cm]{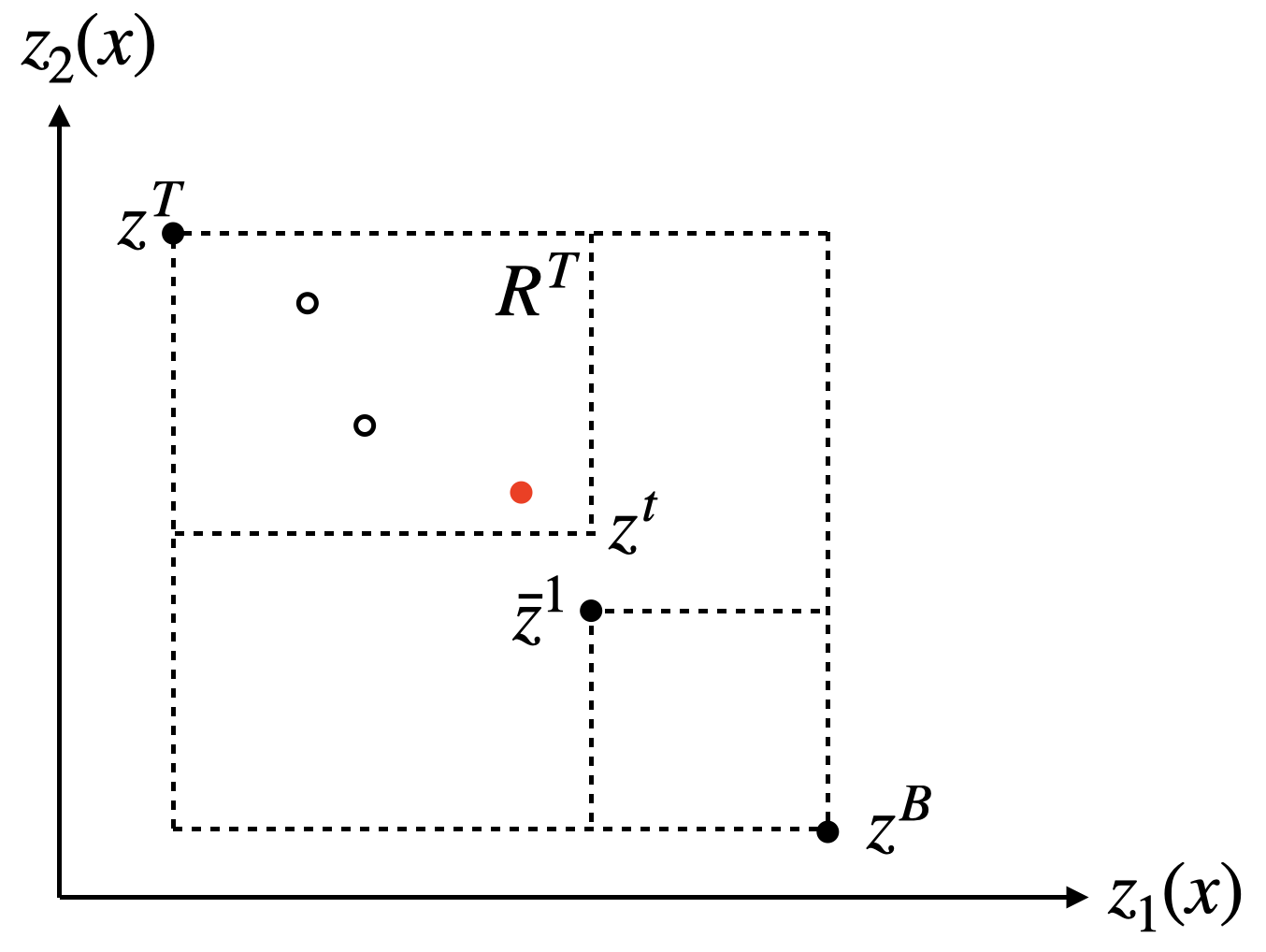}
    \label{balanced_box2}
    }
    \hfill
    \subfigure[Step3.]{
    \includegraphics[width=5cm]{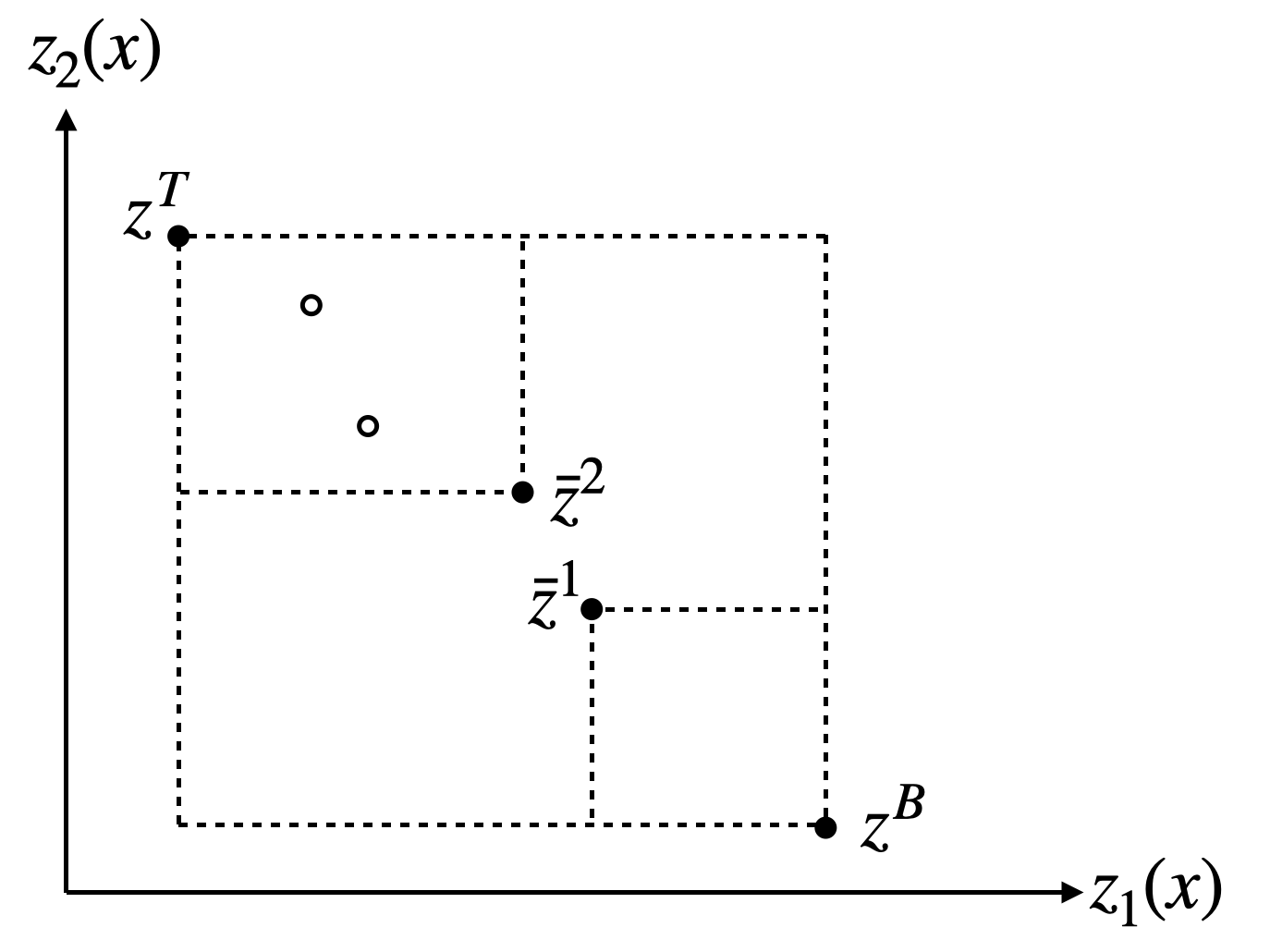}
    \label{balanced_box3}
    }
\caption{\centering Illustration for balanced box method}
\end{figure}

\subsection{Implementation details of the BOBP algorithm} \label{implementation of the BOBP}


The principle of the BOBP algorithm is similar to the single-objective one and the objective is to divide the original problem into easier subproblems which are stored in \textit{``nodes"} (denoted as $\eta$). Each subproblem of the BO-EADARP is denoted as $P(\eta)$. However, the BOBP algorithm is different from the single-objective B\&P in the sense that we calculate lower bound and upper bound sets (instead of single numerical values) to decide whether to fathom a node. Therefore, it is possible that a node is partially fathomed. The main ingredients of the BOBP algorithm are summarized as follows:
\begin{itemize}
    \item Calculate lower bound set and update upper bound set: On each $\eta$, we calculate the lower bound set (denoted as $\mathcal{L}(\eta)$) applying the dichotomic method proposed in \cite{aneja1979bicriteria}. To solve each weighted-sum objective problem, we call the CG algorithm of \cite{su2024branch}. Once $\mathcal{L}(\eta)$) is calculated, we first check if new non-dominated points are obtained. If it is the case, the upper bound set $\mathcal{U}$ is updated.
    \item Lower bound filtering and node fathoming: Then, the obtained $\mathcal{L}(\eta)$) is filtered with $\mathcal{U}$, which returns a set of disjoint non-dominated portions.
    If no such portion is generated after the filtering process, then the analyzed node $\eta$ can be fathomed, as it is fully dominated by the current upper bound set $\mathcal{U}$.
    \item Branching procedure: If the analyzed node cannot be fathomed, branching is applied to generate child nodes. We considered branching strategies proposed in \cite{parragh2019branch} and apply them to each disjoint non-dominated portion. After branching, a set of child nodes are added to the unprocessed node set $\mathcal{T}$.
\end{itemize}
The tree search terminates when there is no unprocessed node remaining in $\mathcal{T}$, and we have the set of non-dominated points $\mathcal{Y}_N$ equals to $\mathcal{U}$. 

The implementation of the BOBP algorithm is organized as follows: we first present the update process of the upper bound set and then introduce the notion of lower bound segments, as in \cite{parragh2019branch}. By defining lower bound segments (abbreviated as LB segments, defined as below), each LB segment included in $\mathcal{L}(\eta)$ at node $\eta$ is compared to all the points in $\mathcal{U}$ and dominated LB segments are removed. Then, the filtered $\mathcal{L}(\eta)$ is served as input to the branching procedure, which includes the objective branching concept proposed in \cite{parragh2019branch} and other branching rules. 

\begin{defn}[LB segments defined in \cite{parragh2019branch}]\label{LB segment}
 Assuming there are two points $p = (p_1,p_2)$, $q = (q_1,q_2)$, and a local nadir $c =(c_1,c_2)$, such that $p_1<q_1$, $p_2 > q_2 $ and $c_1 > q_1$, $c_2 > p_2$. These three points define a segment $s$ if and only if $\{(z_1,z_2)\in \Xi|z_2 < az_1+b\} = \emptyset$, where $a = (q_2-p_2)/(q_1-p_1)$ presents the slope of the line connecting $p$ and $q$ and $b = p_2 - a*p_1$ is the intercept of this line on $z_2$-axis.
\end{defn}

\subsubsection{Updating upper bound set}
We check if there is(are) integer solution(s) generated each time after we calculate $\mathcal{L}(\eta)$ at a node $\eta$. If an integer solution is obtained, $\mathcal{U}$ is updated. 
One of the following cases can happen when we obtain an integer solution $x$.
\begin{enumerate}
    \item Solution $x$ is dominated by solutions corresponding to points in the current upper bound set $\mathcal{U}$. In this case, $x$ is discarded.
    \item Solution $x$ is not dominated by any solutions corresponding to points in the current upper bound set $\mathcal{U}$. In this case, $x$ is added to the set of efficient solutions, and its corresponding point is added to $\mathcal{U}$. Moreover, if solution $x$ dominates some solutions corresponding to points in $\mathcal{U}$, then 
    dominated points are removed from $\mathcal{U}$.
\end{enumerate}

\subsubsection{Lower bound set filtering and node fathoming} \label{sec::lower bound filtering}

One of the main challenges for the bi-objective branch-and-bound algorithm is to evaluate whether a node $\eta$ can be fathomed. It can be done by comparing the obtained $\mathcal{L}(\eta)$ with the current $\mathcal{U}$ and eliminate search areas that will not contain any non-dominated solutions. Assume that we obtain the lower bound set $\mathcal{L}(\eta)$ by solving $P(\eta)$ with dichotomic approach. Each point of $\mathcal{L}(\eta)$ in the objective space corresponds to the fractional solution of a weighted-sum problem.
Then, we sort $\mathcal{L}(\eta)$ in increasing order of the $z_1(x)$ axis. Any two consecutive points in the sorted $\mathcal{L}(\eta)$ accompanying a valid local nadir point define a search area where new non-dominated points may exist. A valid local nadir point can be obtained by simply taking large enough values for both objectives as coordinates. 
The sorted $\mathcal{L}(\eta)$ is then filtered by the current upper bound set $\mathcal{U}$. The detailed filtering process can be also found in Algotithm 1 of \cite{parragh2019branch}.

It should be noted that if there exists non-dominated portions generated after filtering the lower bound set, the node cannot be fathomed, and branching is called. In the case that no portion can be generated after filtering the lower bound set, the node is fathomed.

\subsubsection{Branching procedure} \label{sec:: bi-objective branching procedure}
After filtering the obtained lower bound set $\mathcal{L}(\eta)$ at a node $\eta$, if this node cannot be fathomed, we branch on $\eta$ to generate new child nodes. We consider different branching strategies for each disjoint non-dominated portion: (1) decision space branching and (2) objective space branching, which are adapted from literature (\cite{parragh2019branch,forget2022enhancing}). Interested readers can refer to Algorithm 2 of \cite{parragh2019branch}.

\end{document}